\documentclass[a4paper,10pt]{elsarticle}
\usepackage[a4paper,top=3cm,bottom=3cm,left=3cm,right=3cm,%
bindingoffset=5mm]{geometry}
\usepackage{ucs}
\usepackage[utf8]{inputenc}
\usepackage{amsmath}

\usepackage{amsfonts}
\usepackage{amssymb}
\usepackage{amsthm}
\usepackage{relsize}
\usepackage{caption}
\usepackage{multicol}
\usepackage{multirow}
\usepackage[english]{babel}
\usepackage{fontenc}
\usepackage{graphicx}
\usepackage{todonotes}
\usepackage{color}
\usepackage{enumerate}
\usepackage{bm}
\usepackage[percent]{overpic}

\newtheorem{prop}{Proposition}[section]
\newtheorem{rem}{Remark}[section]

\newcommand{\PiNabla}{\Pi^\nabla_k}

\newcommand{\PiZero}{\Pi^0_k}
\newcommand{\PiGrad}{{\bm \Pi}^0_{k-1}}

\newcommand{\m}[1]{m_{\,#1}}
\newcommand{\mIndex}[2]{m_{(\,#1\,,#2)}}
\newcommand{\mB}{{\bm m}}
\newcommand{\mVect}[2]{{\bm m}_{#1,\,#2}}
\newcommand{\mMat}[4]{{\bm M}_{\,#1,\,#2\,;\,#3,\,#4}}

\newcommand{\mMatB}{{\bm M}}
\newcommand{\xPerp}{{\bm x}^\perp}
\newcommand{\mPerp}{{\bm m}^\perp}
\newcommand{\x}{{\bm x}}

\newcommand{\vvh}{{\bm v}_h}
\newcommand{\uuh}{{\bm u}_h}
\newcommand{\pp}{{\bm p}}
\newcommand{\bb}{{\bm b}}

\newcommand{\ph}{p_h}
\newcommand{\vv}{{\bm v}}
\newcommand{\uu}{{\bm u}}
\newcommand{\ww}{{\bm w}}
\newcommand{\ff}{\bm f}
\newcommand{\nn}{\bm n}

\newcommand{\p}{p}

\newcommand{\MM}{{\bm M}}

\newcommand{\dE}{\,\text{d}E}
\newcommand{\de}{\,\text{d}e}

\newcommand{\aalpha}{{\bm \alpha}}
\newcommand{\bbeta}{{\bm \beta}}
\newcommand{\ggamma}{{\bm \gamma}}
\newcommand{\ddelta}{{\bm \delta}}

\newcommand{\numberset}{\mathbb}
\newcommand{\N}{\numberset{N}}

\newcommand{\Pk}{\numberset{P}}
\renewcommand{\epsilon}{\varepsilon}
\renewcommand{\theta}{\vartheta}
\renewcommand{\rho}{\varrho}
\renewcommand{\phi}{\varphi}

\newcommand{\fg}{\boldsymbol{\phi}}

\newcommand{\dd}{{\rm div}}
\newcommand{\ddVect}{\textbf{div}}
\newcommand{\gr}{\nabla}
\newcommand{\Gr}{\boldsymbol{\nabla}}
\newcommand{\epseps}{\boldsymbol{\epsilon}}
\newcommand{\grepseps}{\mathlarger{\epseps}}

\newcommand{\dl}{\boldsymbol{\Delta}}
\newcommand{\VV}{\mathbf{V}}

\def\P0{{\Pi^{0}_k}}

\def\PN{{\Pi^{\Gr}_k}}
\def\Peps{{\Pi^{\grepseps}_k}}

\def\PP0{{\boldsymbol{\Pi}^{0}_{k-1}}}

\theoremstyle{plain}

\begin{document}

\title{Bricks for the mixed high-order virtual element method:\\
projectors and differential operators}

\author[add1]{F.~Dassi}
\ead{franco.dassi@unimib.it}
\author[add1]{G.~Vacca}
\ead{giuseppe.vacca@unimib.it}

\address[add1]{Department of Mathematics and Applications, University of Milano - Bicocca,\\ 
Via Cozzi 53, I-20153, Milano (Italy)}

\begin{abstract}
We present the essential instruments to deal with Virtual Element Method (VEM)
for the resolution of partial differential equations in mixed form.
Functional spaces, degrees of freedom, projectors and differential operators
are described emphasizing how to build them in a virtual element framework and for a general approximation order.
To achieve this goal, it was necessary to make a deep analysis on polynomial spaces and decompositions.
Finally, we exploit such ``briks'' to construct virtual element approximations of 
Stokes, Darcy and Navier-Stokes problems and we provide a series of examples 
to numerically verify the theoretical behavior of high-order VEM.
\end{abstract}

\begin{keyword}
Virtual Element Method\sep
Mixed Problems \sep
Polygonal meshes \sep
Projectors\sep 
High order.
\MSC[2010] 65N30
\end{keyword}

\maketitle

\section{Introduction}

The virtual element method (VEM) was introduced in \cite{volley, autostoppisti} 
as an extension of finite element method to general polygonal/polyhedral meshes.

The virtual element discrete spaces are similar to the usual polynomial spaces with the addition of suitable (and unknown) non-polynomial functions that are implicitly defined as the solutions of a proper PDE on each element of the decomposition.
The main idea of the  the method is 
to define approximated discrete forms computable \emph{only} via degree of freedom values.
Moreover, it does not approximate non-polynomial test and trial functions at the integration points, but
it exploits some polynomial projections which are \emph{exactly} computed starting from the degrees of freedom.
Using such projections, VEM can handle very general polygonal/polyhedral meshes without the need to integrate complex non-polynomial functions on the elements (as polygonal FEM do) and without loss of accuracy. 
We refer to \cite{volley, projectors, BBMR:2016, BBMRmixed:2016, BBMRserendipity:2016, BLR-stab, brenner:2018} for a in-depth theoretical analysis of the virtual elements features.

The Virtual Element Method has been developed successfully for a large range of mathematical and engineering problems,
we mention, as sample, the very brief list of 
papers~\cite{BBM:2013, BFM:2014, CBP:2017, BMRR:2017, WRRH:2017, ADLP:2017, CBP:2018, BBDMR:2018},
while for the specific topic of implementation aspects related to the VEM 
we mention~\cite{Russo:2016, BPP:2017, BDR:2017, BB:2017, ABSV:2017, BBDMRserendipity:2018, Mascotto:2018, DM:2018, CGMV:2018, BRV:2018}. 
Concernig the mixed PDEs we refer to~\cite{LLC:2017,  CGM:2016, CGS:2018, GMS:2018, CW:2018}
as a sample of VEM papers dealing with such kind of problem,
and to~\cite{LVY:2014, CFQ:2017, BdPD:2018, dPK:2018}
as a representative list of papers treating the same topic with different polytopal technologies. 

In the context of  Stokes or Darcy flows and 
in many physical applications such as 
the models for precipitation and for flows in root-soil
(see for instance \cite{Wheeler, Formaggia, Berardi1, Berardi2} and the references therein), 
the use of general polytopal meshes can be very useful so 
a virtual element approach is particularly appealing. 

The potentiality of the VEM is not limited to the meshing aspect.
Indeed, the flexibility of the virtual element framework has been exploited to construct a $H^1$-conforming virtual element space particularly suited for the mixed problems
~\cite{ABMV:2014,BLV:2017,vacca:2018,BLV:2018,BMV:2018}.
By choosing a suitable pressure space, 
the virtual element approach leads to an \emph{exactly divergence-free} discrete velocity kernel. 
Such property is really important to solve PDEs associated with incompressible fluid flows and 
we further underline that such property is not shared by most of the standard mixed finite element methods, 
where it is imposed only in a weak sense \cite{benchmark, john-linke-merdon-neilan-rebholz:2017}.

In the present contribution we show in detail the practical aspects of the high-order schemes developed in 
\cite{BLV:2017,vacca:2018,BLV:2018}.
We stress that such definition of the virtual elements and the associated degrees of freedom 
is more involved with respect to the \textit{plain-vanilla} $H^1$-conforming VEM space \cite{volley}.
Then, since a \textit{user manual} can be very helpful for people with some experience 
in implementing the Virtual Element scheme, we focus on the explicit construction of ``VEM bricks''
(projectors and discrete forms) to deal with such kind of discretization.
More specifically,
we give the practical instructions
for the computation of the $L^2$--projection, the $\Gr$--projection and the $\epseps$--projection via the degrees of freedom
in the spirit of the hitchhiker's guide~\cite{autostoppisti}.
Moreover, we make a wide variety of numerical tests to show the practical performance of the method for the mixed problems,
underlining the robustness of the scheme with respect to high-order degree of accuracy.

The paper is organized as follows. 
In Section \ref{sec:def} we introduce same definitions and preliminaries and we fix the nations. 
In Section \ref{sec:VEM} we review the family of Virtual Elements 
presented in \cite{vacca:2018} and we introduce the stiffness matrices associated with the mixed problems.
In Section \ref{sec:proj} we explicitly show how to compute the polynomial projections using (as unique information) the degree of freedom values. 
In Section \ref{sec:tests} we analyse the algebraic form of the linear system arising from the virtual element discretization and
we present
several numerical tests which highlight the actual performance of our approach
for the Stokes, the Darcy and the Navier--Stokes equation also for high-order polynomial degree.
Finally, we draw some conclusions.\\

\noindent \textbf{Notation.}~We will follow the usual notation for the differential operators.
Hence the symbols $\gr$ and $\Delta$  denote the gradient and Laplacian for scalar functions, while 
$\dl $, $\Gr$, and $\dd$ denote the vector Laplacian,  the gradient operator
and the divergence  for vector fields, respectively.
Furthermore for a vector field $\vv$ we denote with $\epseps(\vv)$ the symmetric part of the gradient of $\vv$, i.e.
\[
\epseps(\vv) := \frac{\Gr \vv + (\Gr \vv)^{\rm T}}{2} \,.
\]
Finally $\ddVect$ denotes the vector valued divergence for matrix fields.

%
%
%
%
%
%
%
%
%
%
%
%
%
%
%
%
%
%
%
%
%

\section{Definitions \& preliminaries}
\label{sec:def}
In this section we introduce the basic mathematical notation and tools 
to deal with the Virtual Element Method.
From now on let $E \subset\mathbb{R}^2$ be a polygon, we will denote by $\x_E$, $h_E$, $|E|$ the centroid, the diameter and the measure of $E$, respectively.
Let $n \in \N$ and
let $\mathbb{P}_n(E)$, $[\mathbb{P}_n(E)]^2$, $[\mathbb{P}_n(E)]^{2 \times 2}$ be 
the space of scalar, vectorial and matrix polynomials defined on $E$ 
of degree less or equal to $n$, respectively (with the extended notation $\Pk_{-1}(E)=\{0\}$). 
The dimension of such spaces are 
\begin{equation*}
 \dim(\mathbb{P}_n(E)) =   \pi_n := \frac{(n+1)(n+2)}{2}\qquad
\dim([\mathbb{P}_n(E)]^2) = 2 \, \pi_n \qquad
\dim([\mathbb{P}_n(E)]^{2 \times 2}) = 4 \, \pi_n
\,.
\end{equation*}

\vspace{1em}

One of the main tool exploited in the VEM is the so-called \emph{scaled-monomials}. Given a multi-index $\aalpha:=(\alpha_1,\,\alpha_2)$ with  $|\aalpha|=\alpha_1+\alpha_2$, 
a scaled monomial is defined as 
\[
\m{\aalpha} := \left(\frac{x-x_E}{h_E}\right)^{\alpha_1}
\left(\frac{y-y_E}{h_E}\right)^{\alpha_2}\,.
\]
From now on we refer to the null monomial by $\m{\emptyset}$, i.e. $\m{\emptyset}=0$.

With a slight abuse of notation we may denote the scaled monomial $\m{\aalpha}$
with the notation $m_i$, where the relation between the one dimensional index $i$ and the   multi-index $\aalpha$ is given by the natural correspondence
\begin{equation}
1 \mapsto (0, \, 0) \,, \qquad
2 \mapsto (1, \, 0) \,, \qquad 
3 \mapsto (0, \, 1) \,, \qquad
4 \mapsto (2, \, 0) \,, \qquad
5 \mapsto (1, \, 1) \,, \qquad
\dots
\label{eqn:numbering}
\end{equation}
It is easy to show that the set 
\begin{equation}
\mathbb{M}_n(E) :=\left\{ \m{\aalpha}\::\: 0\leq |\aalpha|\leq n\right\}
:= \left\{ m_{i}\::\: 1\leq i \leq \pi_n\right\}\,,
\label{eqn:monoBasis}
\end{equation}
is a basis for $\mathbb{P}_n(E)$. Moreover for any $m \leq n$ we denote with 
\[
\widehat{\mathbb{P}}_{n \setminus m}(E) := {\rm span}(\m{\aalpha}\::\: m+1\leq |\aalpha|\leq n)
\]
i.e. the set of polynomial of degree $n$ which monomials have degree strictly greater than $m$.

\vspace{1em}

The definition of scaled monomial can be extended to the vectorial monomial.
Let $\aalpha:=(\alpha_1,\,\alpha_2)$ and $\bbeta:=(\beta_1,\,\beta_2)$ be two multi-indexes,
then we define a vectorial scaled monomial as
$$
\mVect{\aalpha}{\bbeta} := \binom{\m{\aalpha}}{\m{\bbeta}}\,.
$$
Also in this case, it is easy to show that the set 
\begin{equation}
[\mathbb{M}_n(E)]^2 :=\left\{ \mVect{\aalpha}{\emptyset}\::\: 
0\leq |\aalpha|\leq n\right\}\cup
\left\{\mVect{\emptyset}{\bbeta}\::\: 
0\leq |\bbeta|\leq n\right\}
:= \left\{ \mB_i \::\:1 \leq i \leq  2\,  \pi_n  \right\} \,,\label{eqn:monoBasisVect}
\end{equation}
is a basis for the vectorial polynomial space $[\mathbb{P}_n(E)]^2$, where we implicitly use the natural correspondence between on dimensional indexes and double multi-indexes (extending correspondence \eqref{eqn:numbering}).

\vspace{1em}

\begin{rem}
Note that the following polynomials decomposition holds
$$
[\mathbb{P}_n(E)]^2 = \nabla\,\mathbb{P}_{n+1}(E)\oplus\xPerp\mathbb{P}_{n-1}(E)\,,
$$
where $\xPerp:=(y,-x)^t$.
\\
\noindent
In particular for each ${\bm p}_n\in[\mathbb{P}_{n}(E)]^2$, 
there exist unique $p_{n+1}\in \widehat{\mathbb{P}}_{n+1 \setminus 0}(E) $ and $q_{n-1}\in\mathbb{P}_{n-1}(E)$
such that
\begin{equation}
{\bm p}_n = \nabla\,p_{n+1} + \xPerp\,q_{n-1}\,.
\label{eqn:decomp}
\end{equation}
\label{prop:decomp}
\end{rem}

The decomposition in Remark~\ref{prop:decomp} is essential 
to define projector operators
and consequently to proceed with a virtual element analysis 
for a large variety of PDEs. 
Unfortunately, finding such decomposition for a generic vectorial polynomial ${\bm p}_n$ is not an easy task, but, if we consider scaled monomials,
we found a straightforward recipe to get it.

\begin{prop}
Consider the  vectorial monomials 
$\mVect{\aalpha}{\emptyset}$ and $\mVect{\emptyset}{\bbeta}\in[\mathbb{M}_n(E)]^2$, 
with 
$\aalpha =(\alpha_1,\,\alpha_2)$ and $\bbeta =(\beta_1,\,\beta_2)$.
Then referring to \eqref{eqn:decomp}, the following scaled decompositions hold
\begin{gather}
\mVect{\aalpha}{\emptyset} := \frac{h_E}{|\aalpha|+1}\nabla \mIndex{\alpha_1+1}{\alpha_2} + 
\frac{\alpha_2}{|\aalpha|+1} \mPerp \mIndex{\alpha_1}{\alpha_2-1}\,,
\label{eqn:decompX}
\\
\mVect{\emptyset}{\bbeta} := \frac{h_E}{|\bbeta|+1}\nabla \mIndex{\beta_1}{\beta_2+1} - 
\frac{\beta_1}{|\bbeta|+1} \mPerp \mIndex{\bbeta_1-1}{\beta_2}\,,
\label{eqn:decompY}
\end{gather}
where $\mPerp := (\m{(0,1)},\,-\m{(1,0)})^t$.
\label{prop:decompMono}
\end{prop}

\begin{proof}
We show the decomposition in Equation~\eqref{eqn:decompX},
the one in \eqref{eqn:decompY} follows the same strategy.
We compute the following quantities
\begin{equation}
\nabla \mIndex{\alpha_1+1}{\alpha_2} = \frac{1}{h_E}
\binom{(\alpha_1+1)\,\mIndex{\alpha_1}{\alpha_2}}
{\alpha_2\,\mIndex{\alpha_1+1}{\alpha_2-1}}
\label{eqn:firstPartProp}
\end{equation}
and
\begin{equation}
\mPerp \mIndex{\alpha_1}{\alpha_2-1} = 
\binom{\mIndex{\alpha_1}{\alpha_2}}{-\mIndex{\alpha_1+1}{\alpha_2-1}}\,.
\label{eqn:secondPartProp}
\end{equation}
Note that the coefficient $1/h_E$  in \eqref{eqn:firstPartProp}
is due to the chain derivative rule.
Therefore the choice of the multi-indexes  on the right hand side of Equation~\eqref{eqn:decompX} 
produces two vectorial polynomials
with the same monomial at the same components.
Notice that, according to \eqref{eqn:decomp}, the gradient component of the decomposition has strictly positive degree.  
Decomposition \eqref{eqn:decompX} comes from a proper linear combination of~\eqref{eqn:firstPartProp} and~\eqref{eqn:secondPartProp}.
\end{proof}

\begin{rem}
Proposition~\ref{prop:decompMono}
is an easy tool to compute the decomposition \eqref{eqn:decomp} for any general polynomial ${\bm p}_n \in [\mathbb{P}_n(E)]^2$
and represents a key ingredient in the implementation of the proposed schemes.
\end{rem}

Finally we consider the matrix polynomial space $[\mathbb{P}_n(E)]^{2\times 2}$
and we define the matrix scaled monomials
\begin{equation*}
\mMat{\aalpha}{\bbeta}{\ggamma}{\ddelta}:=
\left(
\begin{array}{cc}
m_{\aalpha} &m_{\bbeta}\\
m_{\ggamma} &m_{\ddelta}\\
\end{array}
\right)\,,
\end{equation*}
where $\aalpha,\,\bbeta,\,\ggamma$ and $\ddelta$ are multi-indexes.
We build a basis of $[\mathbb{P}_n(E)]^{2\times 2}$ 
starting from matrix scaled monomials  in the natural way (using again the usual correspondence between indexes and multi-indexes)
\begin{equation}
\begin{aligned}
{[\mathbb{M}_n(E)]}^{2 \times 2} :=& 
\left\{ \mMat{\aalpha}{\emptyset}{\emptyset}{\emptyset}\::\: 
0\leq |\aalpha|\leq n\right\} \cup 
\left\{\mMat{\emptyset}{\bbeta}{\emptyset}{\emptyset}\::\: 
0\leq |\bbeta|\leq n\right\} \cup  \\ 
&  \left\{ \mMat{\emptyset}{\emptyset}{\ggamma}{\emptyset}\::\: 
0\leq |\ggamma|\leq n\right\} \cup 
\left\{\mMat{\emptyset}{\emptyset}{\emptyset}{\ddelta}\::\: 
0\leq |\ddelta|\leq n\right\} 
\\
:=& \left\{\MM_i \::\: 
 1\leq i \leq 4 \, \pi_n \right\} 
\end{aligned}
\label{eqn:monoBasisMat}
\end{equation}
that clearly is a basis for  $[\mathbb{P}_n(E)]^{2\times 2}$.

\vspace{3ex}

\begin{rem}
Note that in Definitions \eqref{eqn:monoBasis}, \eqref{eqn:monoBasisVect} and \eqref{eqn:monoBasisMat} we consider both the index and the multi-index notations. Both notations  will be employed indifferently when we are dealing with scaled monomials.
\end{rem}


%
A key ingredient in the VEM construction is represented by the polynomial projections that will play a fundamental role in the construction of the approximated virtual elements form.
For any $n \in \N$ 
we define the following polynomial projections:
\begin{itemize}
\item the $\boldsymbol{L^2}$\textbf{--projection} $\Pi_n^{0} \colon [L^2(E)]^2 \to [\Pk_n(E)]^2$, defined for all $\vv \in [L^2(E)]^2$ by
\begin{equation}
\label{eq:P0_k^E}
\int_E \mathbf{q}_n \cdot (\vv - \, {\Pi}_{n}^{0}  \vv) \, {\rm d} E = 0 \qquad  \text{for all $\mathbf{q}_n \in [\Pk_n(E)]^2$,} 
\end{equation} 
with obvious extension 
for matrix functions 
$\boldsymbol{\Pi}_{n}^{0} \colon [L^2(E)]^{2 \times 2} \to [\Pk_{n}(E)]^{2 \times 2}$,

\item the $\boldsymbol{\nabla}$\textbf{--projection} ${\Pi}_{n}^{\Gr} \colon [H^1(E)]^2 \to [\Pk_n(E)]^2$, defined for all $\vv \in [H^1(E)]^2$ by
\begin{equation}
\label{eq:Pn_k^E}
\left\{
\begin{aligned}
& \int_E \Gr  \, \mathbf{q}_n :  \Gr ( \vv - \, {\Pi}_{n}^{\Gr}   \vv) \, {\rm d} E = 0 &\qquad  &\text{for all $\mathbf{q}_n \in [\Pk_n(E)]^2$,} \\
& \int_{\partial E}\mathbf{q}_0 : (\vv - \,  {\Pi}_{n}^{\Gr}  \vv) \de= 0 &\qquad  &\text{for all $\mathbf{q}_0 \in [\Pk_0(E)]^2$,} \\
\end{aligned}
\right.
\end{equation}

\item the $\boldsymbol{\epsilon}$\textbf{--projection} ${\Pi}_{n}^{\grepseps} \colon [H^1(E)]^2 \to [\Pk_n(E)]^2$, defined for all $\vv \in [H^1(E)]^2$ by
\begin{equation}
\label{eq:Peps_k^E}
\left\{
\begin{aligned}
& \int_E \epseps   (\mathbf{q}_n) :  \epseps ( \vv - \, {\Pi}_{n}^{\grepseps}   \vv) \, {\rm d} E = 0 &\qquad  &\text{for all $\mathbf{q}_n \in [\Pk_n(E)]^2$,} \\
& \int_{\partial E}\mathbf{q} : (\vv - \,  {\Pi}_{n}^{\grepseps}  \vv) \de= 0 &\qquad  &\text{for all $\mathbf{q}\in \langle [\Pk_0(E)]^2, \x^{\perp}   \rangle $.} \\
\end{aligned}
\right.
\end{equation} 
\end{itemize}

\section{The Virtual Element Approximation}\label{sec:VEM}

 In the present section we summarize a  short overview of the $H^1$--conforming Virtual
 Elements for the mixed problems. 
 We will make use of various tools from the virtual element technology,
 that will be described briefly. 
 We refer the reader to~\cite{BLV:2017,vacca:2018,BLV:2018} for a deeper analysis.
 
 Let $\{\Omega_h\}_h$ be a sequence of decompositions of $\Omega$ into general polygonal elements $E$
with
\[
h := \sup_{E \in \Omega_h} h_E .
\]
For all $h$ we suppose that  each element $E$ in $\Omega_h$ fulfils the following assumptions:
\begin{description}
\item [$\mathbf{(A1)}$] $E$ is star-shaped with respect to a ball $B_E$ of radius $ \geq\, \rho \, h_E$, 
\item [$\mathbf{(A2)}$] the distance between any two vertexes of $E$ is $\geq \rho \, h_E$, 
\end{description}
where $\rho$ is a uniform positive constant. We remark that the hypotheses above, though not too restrictive in many practical cases, 
can be further relaxed, as investigated in ~\cite{BLR-stab, brenner:2018}.

\subsection{Virtual elements spaces}
\label{sub:spaces}  
 Let $k \geq 2$ the polynomial degree of accuracy of the method. We consider  on each element $E \in \Omega_h$ the (enlarged) finite dimensional local virtual space
 \begin{multline*}
 \mathbf{U}_h^E := \biggl\{  
 \vv \in [H^1(E) \cap C^0(E)]^2 \quad \text{s.t} \quad {\vv}_{|_e} \in [\Pk_k(e)]^2 \quad \text{for all $e \in \partial E$} \, , \biggr.
 \\
 \left.
 \biggl\{
 \begin{aligned}
 & - \boldsymbol{\Delta}    \vv  -  \nabla s \in \mathbf{x}^{\perp}\, \Pk_{k-1}(E),  \\
 & {\rm div} \, \vv \in \Pk_{k-1}(E),
 \end{aligned}
 \biggr. \qquad \text{ for some $s \in  L^2_0(E)$}
 \quad \right\}.
 \end{multline*}
 Referring to \cite{vacca:2018} we introduce the Virtual Element space $\VV_h^E$
 as the restriction  of $\mathbf{U}_h^E$ defined by:
 \begin{equation}
 \label{eq:V_h^E}
 \VV_h^E := \left\{ \vv \in \mathbf{U}_h^E \quad \text{s.t.} \quad   \left(\vv - \PN \vv, \, \mathbf{x}^{\perp}\, \widehat{p}_{k-1} \right)_{E} = 0 \quad \text{for all $\widehat{p}_{k-1} \in  \widehat{\Pk}_{k-1 \setminus k-3}(E)$} \right\} \,.
 \end{equation}
 We here summarize the main properties of the virtual
 space $\VV_h^E$  (we refer \cite{vacca:2018,BLV:2018} for a detailed analysis):
 \begin{itemize}
 \item \textbf{dimension:}
  the dimension of $\VV_h^E$  is
  \begin{equation*}
 \begin{split}
 \dim\left( \VV_h^E \right) 
 &= 2n_E \, k + \frac{(k-1)(k-2)}{2}  + \frac{(k+1)k}{2} - 1
 \end{split}
 \end{equation*}
 where $n_E$ is the number of vertexes of $E$;
 \item \textbf{degrees of freedom:}
 Let $\texttt{NDoF} := \dim(\VV_h^E)$, the linear operators 
 $$\mathbf{D_V} := \{ \mathbf{D}_{\mathbf{V}, i} \}_{i=1}^{\texttt{NDoF}} \,,$$ 
 split into four subsets (see Figure \ref{fig:dofsloc}) constitute a set of DoFs for $\VV_h^E$:
 \begin{itemize}
 \item ${\mathbf{D}^{\rm vertex}}$:  the values of $\vv$ at the vertexes of the polygon $E$,
 \item ${\mathbf{D}^{\rm edge}}$: the values of $\vv$ at $k-1$ distinct  internal points of the $(k+1)$-point Gauss--Lobatto rule on  every edge $e \in \partial E$,
 \item $\mathbf{D}^{\mathbf{m}^{\perp}}$: the moments of $\vv$
 \[
\mathbf{D}^{\mathbf{m}^{\perp}}_{i} (\vv)
 :=
\frac{1}{|E|} \int_E \vv \cdot \mathbf{m}^{\perp}\, m_{i} \, {\rm d}E \qquad \text{for  $1 \leq i\leq \pi_{k-3}$,}
 \]
 \item $\mathbf{D^{\rm div}}$: the moments of ${\rm div} \,\vv$ 
 \[
 \mathbf{D}^{\rm div}_{i} (\vv) := 
\frac{h_E}{|E|} \int_E ({\rm div} \,\vv) \, m_{i} \, {\rm d}E \qquad \text{for  $2 \leq i \leq \pi_{k-1}$;}
 \] 
 \end{itemize}
 \item \textbf{projections:}
 referring to \eqref{eq:P0_k^E}, \eqref{eq:Pn_k^E} and \eqref{eq:Peps_k^E}, 
 the DoFs $\mathbf{D_V}$ allow us to compute exactly the polynomial projections 
 \begin{itemize}
 \item $\PN \colon \VV_h^E \to [\Pk_k(E)]^2$ (see Subsection~\ref{sub:piNabla}),
 \item $\Peps \colon \VV_h^E \to [\Pk_k(E)]^2$ (see Subsection~\ref{sub:piEps}),
 \item $\P0 \colon \VV_h^E \to [\Pk_k(E)]^2$ (see Subsection~\ref{sub:piZero}),
 \item $\PP0 \colon \Gr(\VV_h^E) \to [\Pk_{k-1}(E)]^{2 \times 2}$ (see Subsection~\ref{sub:piGrad}),
 \item $\PP0 \colon \epseps(\VV_h^E) \to [\Pk_{k-1}(E)]^{2 \times 2}$ (see Remark~\ref{rem:piEpsGrad}),
 \end{itemize}
 in the sense that, given any $\vv \in \VV_h^E$, we are able to compute the polynomials
 \[
 \PN \vv \,, \qquad \Peps \vv\,, \qquad
 \P0 \vv\,,  \qquad \PP0\Gr\vv\,, \qquad  \PP0\epseps(\vv)
 \] 
 only using, as unique information, the degree of freedom values $\mathbf{D_V}$ of $\vv$.
 \end{itemize}
 \begin{figure}[!h]
 \center{
 \includegraphics[width=0.25\textwidth]{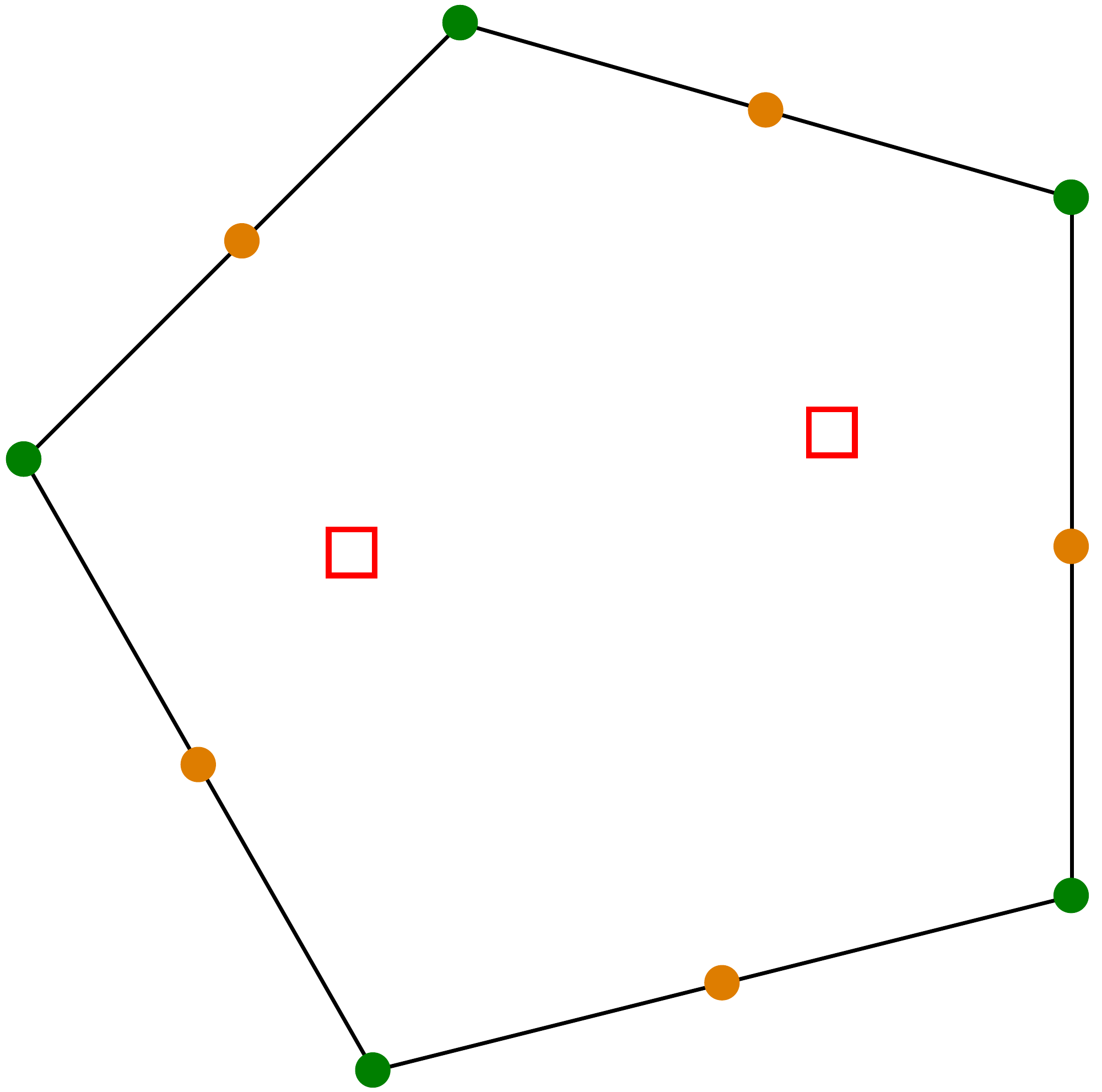} 
 \qquad\qquad\qquad
 \includegraphics[width=0.25\textwidth]{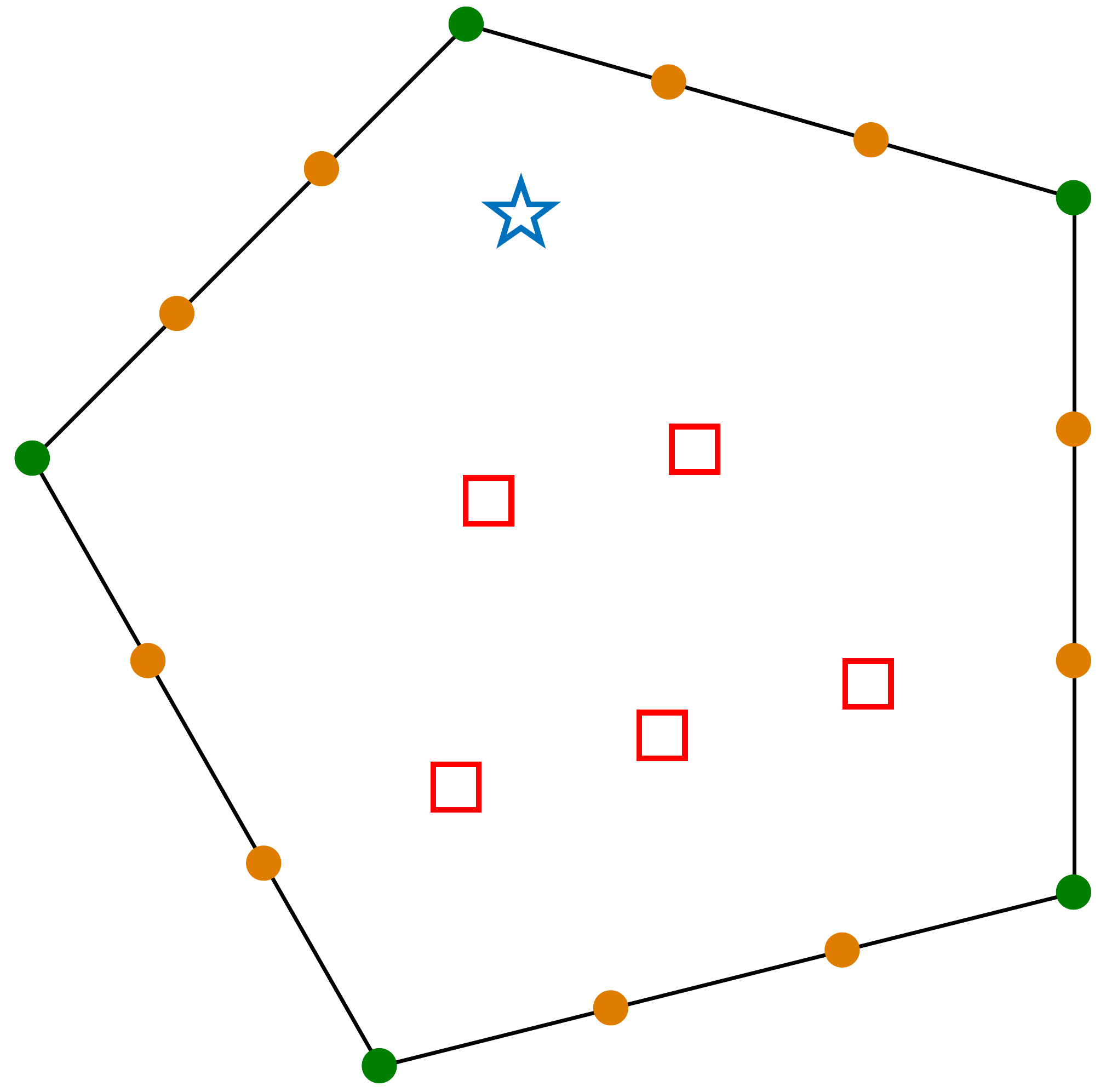}
 \caption{DoFs for $k=2$ (left), $k=3$ (right). We denote $\mathbf{D}^{\rm vertex}$ with green dots, $\mathbf{D}^{\rm edge}$ with orange dots, $\mathbf{D}^{\mathbf{m}^{\perp}}$ with blue stars, $\mathbf{D}^{\rm div}$ with red squares.}
 \label{fig:dofsloc}
 }
 \end{figure}
For future reference, we collect all the $2 k n_E$ boundary DoFs (the first two items above) and
denote them with $\mathbf{D^{ \partial}}
:=\{\mathbf{D}^{\partial}_i\}_{i=1}^{2 k n_E} $. 
Moreover we denote with $\mathbf{D}^e
:=\{\mathbf{D}^{ e}_i\}_{i=1}^{2 k}$ the DoFs of $\mathbf{D^{ \partial}}$
relative to the (closed) edge $e$.
Note that in the case $k=2$ the set of DoFs $\mathbf{D}^{\mathbf{m}^{\perp}}$ is empty.

The basis functions $\fg_i \in \VV_h^E$ are defined as usual as the canonical basis functions
\[
\mathbf{D}_{\mathbf{V}, j} (\fg_i) = \delta_{ji}
\qquad 
\text{for $i, j = 1, \dots, \texttt{NDoF}$,}
\]
in particular for any $\vv \in \VV_h^E$ we have the Lagrange identity
\[
\vv = \sum_{i=1}^{\texttt{NDoF}} \mathbf{D}_{\mathbf{V}, i} (\vv) \, \fg_i \,.
\]

 The global virtual element space is obtained  by combining the local spaces $\VV_h^E$
 accordingly to the local degrees of freedom, as in standard  finite elements and considering the homogeneous boundary conditions. Therefore we get
 \begin{equation}
 \label{eq:V_h}
 \VV_h := \{ \vv \in [H^1_0(\Omega)]^2  \quad \text{s.t} \quad \vv_{|_E} \in \VV_h^E  \quad \text{for all $E \in \Omega_h$} \} \,.
 \end{equation}

 \vspace{2em}
 \noindent The pressure space is simply given by the piecewise polynomial functions
 \begin{equation}
 \label{eq:Q_h}
 Q_h := \{ q \in L_0^2(\Omega) \quad \text{s.t.} \quad q_{|_E} \in
 \Pk_{k-1}(E) \quad \text{for all $E \in \Omega_h$}\} \,,
 \end{equation}
 with  local  DoFs:
 \begin{itemize}
 \item $\mathbf{D_Q}$: the coefficients $\{c_{i}\}_{i=1}^{\pi_{k-1}}$ of $q_{|_E}$ with respect to the re-scaled basis $\mathbb{M}_{k-1}(E)$, i.e.
 \[
 q = \sum_{i=1}^{\pi_{k-1}} c_{i} \, \frac{h_E}{|E|}  m_{i} \,.
 \]

 \end{itemize}

\begin{rem} 
 As observed in \cite{vacca:2018, BLV:2018}, by definitions \eqref{eq:V_h} and \eqref{eq:Q_h}, it holds 
 \begin{equation*}
 {\rm div}\, \VV_h\subseteq Q_h \,.
 \end{equation*}
This entails a series of important advantages: 
the  numerical scheme leads to an exactly
divergence-free discrete velocity solution for incompressible fluid. 
Moreover the proposed family of virtual elements is uniformly stable both for the Darcy and the Stokes problem.
\end{rem} 
 %
 \subsection{Discrete forms and load term approximation}
 \label{sub:fors}

In this subsection we briefly describe the construction of a discrete version of the local stiffness matrices arising from the  mixed problems  such as  Stokes, Darcy and the Navier--Stokes Equations (presented in Section \ref{sec:tests}).
These physical problems share the same algebraic structure once a discretization is introduced. 
In particular we need to define the following stiffness matrices:
\begin{align}
\label{eq:a0}
\left({\bm K}^{0, E}_h \right)_{ij} = 
a_h^{0, E}(\fg_i,\, \fg_j) &\approx \int_E \fg_i \cdot \fg_j \, {\rm d}E
& \qquad 
&\text{for $i, j = 1, \dots, \texttt{NDoF}$,}
\\
\label{eq:aGr}
\left({\bm K}^{\Gr, E}_h \right)_{ij} = 
a_h^{\Gr, E}(\fg_i,\, \fg_j) &\approx \int_E \Gr \fg_i : \Gr \fg_j \, {\rm d}E
& \qquad 
&\text{for $i, j = 1, \dots, \texttt{NDoF}$,}
\\
\label{eq:aeps}
\left({\bm K}^{\grepseps, E}_h \right)_{ij} = 
a_h^{{\grepseps}, E}(\fg_i,\, \fg_j)
&\approx \int_E \epseps (\fg_i) : \epseps (\fg_j) \, {\rm d}E
& \qquad 
&\text{for $i, j = 1, \dots, \texttt{NDoF}$,}
\end{align} 
arising from the Darcy  and the Stokes equation (in the ``gradient form'' or in the ``epsilon form'') respectively.
Following the standard procedure in VEM letterature \cite{volley,BLV:2017,vacca:2018,BLV:2018}, we introduce the computable discrete local bilinear forms:
\begin{align}
\label{eq:ah0}
a_h^{0, E}(\fg_i, \fg_j) &:= 
\int_E \P0 \fg_i \cdot  \P0 \fg_j \,{\rm d}E \, + \,
|E| \, \mathcal{S}\left((I - \P0) \fg_i, \, (I -\P0) \fg_j \right)
\\
\label{eq:ahGr}
a_h^{\Gr, E}(\fg_i, \fg_j) &:= 
\int_E \Gr(\PN \fg_i) :  \Gr(\PN \fg_j) \,{\rm d}E \, + \, 
\mathcal{S}\left((I - \PN) \fg_i, \, (I -\PN) \fg_j \right)
\\
\label{eq:aheps}
a_h^{\grepseps, E}(\fg_i, \fg_j) &:= 
\int_E \epseps(\Peps \fg_i) :  \epseps(\Peps \fg_j) \,{\rm d}E  \, + \,
\mathcal{S}\left((I - \Peps) \fg_i, \, (I -\Peps) \fg_j \right)
\end{align}
where
\begin{equation}
\label{eq:s}
\mathcal{S}(\uu_h, \, \vv_h) := 
\sum_{i=1}^{\texttt{NDoF}} \mathbf{D}_{\mathbf{V}, i}(\uu_h) \, \mathbf{D}_{\mathbf{V}, i}(\vv_h)
\end{equation}
is the inner product of the vectors containing the DoFs values of $\uu_h$ and $\vv_h$ respectively.

Let $a_h^E(\cdot, \, \cdot)$ be one of the discrete bilinear forms listed above.
It is straightforward to check  that the definitions 
\eqref{eq:P0_k^E}, \eqref{eq:Pn_k^E}, \eqref{eq:Peps_k^E} 
and~\eqref{eq:s} imply that the discrete forms $a_h^E(\cdot, \cdot)$
satisfies the consistency and stability properties \cite{volley}.

As usual we define the global approximated bilinear form by adding the local contributions:
 \begin{equation}
 \label{eq:ahglobal}
 a_h(\uu_h, \vv_h) := \sum_{E \in \Omega_h}  a_h^E(\uu_h, \vv_h), \qquad \text{for all $\uu_h, \vv_h \in \VV_h$.}
 \end{equation}

For the treatment of the Navier--Stokes equation we also need to define a computable approximation of the convective trilinear form 
\[
c_h^{E}(\ww_h; \, \uu_h, \vv_h) \approx 
\int_E \left(  \Gr \uu_h \,  \ww_h \right) \cdot \vv_h \, {\rm d}E \,.
\] 
 Referring to \eqref{eq:P0_k^E}  we set
 for all $\ww_h, \uu_h, \vv_h \in \VV_h^E$:
 \begin{align*}
 c_h^E(\ww_h; \, \uu_h, \vv_h) &:= \int_E \left[ \left(\PP0 \, \Gr \uu_h  \right)  \left(\P0 \ww_h \right) \right] \cdot \P0 \vv_h \, {\rm d}E 
 \end{align*}
 and note that all quantities in the previous formula are computable. 
 The global approximated trilinear form is defined by simply
summing the local contributions:
 \begin{equation}
 \label{eq:c_h}
 c_h(\ww_h; \, \uu_h, \vv_h) := \sum_{E \in \Omega_h}  c_h^E(\ww_h; \, \uu_h, \vv_h), \qquad \text{for all $\ww_h, \uu_h, \vv_h \in \VV_h$.}
 \end{equation}
Finally, for any given function $\ff \in [L^2(E)]^2$, we introduce the
computable approximation of the right-hand side
\[
(\ff_h, \, \vv_h)_E \approx (\ff, \, \vv_h)_E  \qquad \text{for all $\vv_h \in \VV_h^E$}
\]
by taking 
\begin{equation}
\label{eq:fh}
{\ff_h}_{|_E} := \P0 \ff \,.
\end{equation}
Therefore for any loads $\ff \in [L^2(E)]^2$ and $g \in L^2(E)$,  the vectors of  the local load terms are simply defined by \cite{vacca:2018,BLV:2018}
 \begin{equation}
 \begin{aligned}
 \label{eq:FG_h}
 \left({\bm f}_h^E \right)_i &:= \int_E  \P0 \ff \cdot  \fg_i \dE 
 &\qquad 
 &\text{for $i=1, \, \dots \,, \texttt{NDoF}$,} \\
 \left({\bm g}^E \right)_{\ell} &:= \frac{h_E}{|E|}\int_E g \, m_{\ell} \dE 
 &\qquad 
 &\text{for ${\ell}=1, \, \dots \,, \pi_{k-1}$.} 
 \end{aligned} 
 \end{equation}

\subsection{Divergence of functions in $\VV_h^E$}
\label{sub:div}

In this subsection we show how to exactly determine the divergence of a function $\vvh\in\VV_h^E$ starting from its degrees of freedom.
First of all we recall from \eqref{eq:V_h^E} that $\dd \, \vvh \in\mathbb{P}_{k-1}(E)$, therefore
we can write this polynomial on the scaled monomial basis $\mathbb{M}_{k-1}(E)$, i.e.
$$
\dd \, \vvh = \sum_{i=1}^{\pi_{k-1}} d_{i} \, \m{i}\,.
$$
To find the unknown coefficients $d_{i}$,
we test $\dd \, \vvh$ against all the elements $m_j$ of $\mathbb{M}_{k-1}(E)$, obtaining 
\begin{equation}
\sum_{i=1}^{\pi_{k-1}} d_{i} \int_E \m{i}\,\m{j}\dE  = 
\int_E (\dd \, \vvh)\,\m{j} \dE\qquad
\text{for  $j=1, \dots, \pi_{k-1}$.}
\label{eqn:divCond}
\end{equation}
All terms in Equation~\eqref{eqn:divCond} are computable although the function $\vvh$ is virtual.
Indeed, the left-hand side is an integral of monomials,
while the right-hand sides are:
\begin{itemize}
 \item if $j \geq 2$ by definition of $\mathbf{D}^{\rm div}$:
 \[
 \int_E (\dd \, \vvh)\,\m{j} \dE = \frac{|E|}{h_E}\, \mathbf{D}^{\rm div}_{j} (\vv_h)
 \]
 \item if $j=1$, by the divergence theorem follows that
 $$
 \int_E (\dd \,  \vvh)\dE = \int_{\partial E} \vvh\cdot\nn\de = 
 \sum_{e\in \partial E}\int_e \vvh\cdot\nn_e\de
 $$
 so it is computable via the DoFs values $\mathbf{D}^e(\vv_h)$ that are exactly the values of $\vv_h$ at the $(k+1)$ Gauss--Lobatto quadrature points on each edge $e \in \partial E$.
\end{itemize}

\begin{rem}
We stress that the algebraic form of Equation~\eqref{eqn:divCond} consists of  a linear system $A{\bm d}=\bb$,
where 
\begin{equation*}
A := \int_E \m{i}\,\m{j}\dE\,,\quad \bb := \int_E (\dd \,\vvh)\,\m{j}\dE\quad
\text{and}\quad {\bm d} := \left(d_{i} \right)_{i=1}^{\pi_{k-1}}\,. 
\end{equation*}
The same algebraic structure: matrix which elements are integrals of polynomials, right hand-side consisting of computable integrals involving the virtual functions, will be at the basis of the computations of all projections  in Section \ref{sec:proj}.
\label{rem:system}
\end{rem}

The argument above give us a recipe to compute exactly the ``divergence form''
involved in the classic velocity-pressure problem, i.e.
\begin{equation}
\label{eq:bform}
b(\vv_h, \, q_h) := \int_{\Omega} (\dd \,\vvh)\,q_h\dE 
\qquad 
\text{for all $\vv_h \in \VV_h$ and $q_h \in Q_h$.}
\end{equation}
Indeed we get
\[
b(\vv_h, \, q_h) =
\sum_{E \in \Omega_h} \int_{E} (\dd \,\vvh)\,q_h \dE  \,
\]
that is explicitly computable.
Moreover, direct computation and the definition of the bases functions for the spaces $\VV_h^E$ and $Q_h^E$, show that the local matrix 
\[
\left({\bm B}^E \right)_{i{\ell}} = \frac{h_E}{|E|}\int_E (\dd \, \fg_i) \, m_{\ell} \dE
\qquad 
\text{for $i=1, \dots, \texttt{NDoF}$ and ${\ell}=1, \dots, \pi_{k-1}$}
\]
has the simple form
\begin{equation}
{\bm B}^E :=
\left[\begin{array}{@{}c|c@{}|c@{}}
  \begin{matrix}
  b_1 &
  \cdots   &
  b_{2 k n_E}
  \end{matrix}
  &
  \begin{matrix}
  0 &
  \cdots   &
  0
  \end{matrix} \, \,
  &
  \begin{matrix}
  0 &
  \cdots   &
  0
  \end{matrix}
  \\
\hline
  \begin{matrix}
   &  & \\ 
   & {\bm 0} & \\
   &   &
  \end{matrix}
  &
  \begin{matrix}
   &  & \\ 
   & {\bm 0} & \\
   &   &
  \end{matrix}
  & 
  \begin{matrix}
   &  & \\ 
   & {\bm I} & \\
   &   &
  \end{matrix}
\end{array}\right]
\label{eq:bmatrix}
\end{equation}
where
\[
b_i = \frac{h_E}{|E|}\int_{\partial E} \fg_i \cdot {\bm n} \de \qquad
\text{for $i=1, \dots , 2 k n_E$.}
\]

\section{How to make projectors}\label{sec:proj}

In this section we focus on the definition of the projection operators described in 
Section~\ref{sec:def}. 
In particular we will exploit how to 
to compute such projections via the degrees of freedom
even if we are dealing with virtual functions.


Let $\Pi_k \colon \VV_h^E \to \Pk_k(E)$ denote one of the projections \eqref{eq:P0_k^E}, \eqref{eq:Pn_k^E} and \eqref{eq:Peps_k^E}.
For a given virtual function $\vv_h \in \VV_h^E$, since by definition $\Pi_k \vvh$ is a vector valued polynomial of degree $k$,
it can be written in terms of the monomial basis $[\mathbb{M}_k(E)]^2$, i.e.
\begin{equation}
\Pi_k \vvh = \sum_{i=1}^{2\pi_k} \zeta_i\,\mB_i\,.
\label{eqn:PiNablaSviluppo}
\end{equation}
Once we find the unknown coefficients $\zeta_i$, 
we uniquely determine such projection.

We further underline that a generic function $\vvh\in\VV_h^E$ is a continuous vectorial polynomial of degree $k$ on each edge in $\partial E$, i.e. ${\vvh}_{|e} \in[\mathbb{P}_k(e)]^2$.
Such polynomial is uniquely determined by $\mathbf{D}^{e}$.
Consequently \emph{each time} we are considering a virtual function $\vvh$ on $\partial E$, 
it has to be considered as a known function.

\begin{rem}
We stress  that the choice of the Gauss--Lobatto  quadrature points on each edge as DoFs
is particularly convenient:  we can compute the
integral of a polynomial of degree $2k-1$
on each edge $e$ directly from its $k+1$ degrees of freedom $\mathbf{D}^e$, and this feature will greatly simplify the  implementation of the method.
\label{rem:boundary1}
\end{rem}

\subsection{$\Gr$--projection $\PN$}\label{sub:piNabla}

Let us start our analysis with the $\Gr$-projection $\PN$ (cf. Definition \eqref{eq:Pn_k^E}).
The target is to determine the coefficients  $\zeta_i$ in \eqref{eqn:PiNablaSviluppo} relative to $\PN$.
To achieve this goal, we replace the generic vectorial polynomial $\pp_k$ 
in Equation~\eqref{eq:Pn_k^E} by monomials in $[\mathbb{M}_k(E)]^2$, obtaining the equivalent system
\begin{equation}
\left\{
\begin{array}{rlll}
\mathlarger{\int_E} \Gr (\vvh - \PiNabla \vvh) \, : \, \Gr\mB_j\dE &=& 0 &
\text{for all $\mB_j\in [\mathbb{M}_k(E)]^2 \setminus [\mathbb{M}_0(E)]^2$,}\\
\mathlarger{\int_{\partial E}} (\vvh - \PiNabla \vvh)\cdot \mB_j\de &=& 0 &
\text{for all $\mB_j\in [\mathbb{M}_0(E)]^2$.}
\end{array}
\right.
\label{eqn:piNablaMono}
\end{equation}
It is easy to show that Equation~\eqref{eqn:piNablaMono} is a set of linearly independent equations which uniquely determine the coefficients $\zeta_i$.
Moreover, Equation~\eqref{eqn:piNablaMono} can be seen 
as a linear system in the unknowns $\zeta_i$ 
in a similar way we have done in Remark \ref{rem:system}
to find the divergence of $\vvh$.
To find such polynomial coefficients,
we have to understand if it is possible to compute the quantities where the virtual function $\vvh$ appears.
Since the second condition in Equation~\eqref{eqn:piNablaMono}
involves integral over the boundary of the polygon,
it is clear that is computable.
Let us consider the first one.
Substituting the definition in Equation~\eqref{eqn:PiNablaSviluppo} and 
moving the virtual function on the right-hand side, we get
$$
\sum_{i=1}^{2\pi_k} \zeta_i\,
\int_E \Gr \mB_i\, : \,\Gr\mB_j\dE  = 
\int_E \Gr \vvh \, : \,\Gr\mB_j\dE\,.
$$
The left-hand side is computable 
since it involves integral of vectorial monomials over the polygon $E$. 
We have to show that the right-hand side is also computable 
via the degrees of freedom of $\vvh$. 
Integrating by parts and we obtain
\begin{eqnarray}
\int_E \Gr \vvh \, : \,\Gr\mB_j\dE &=&
- \int_E \vvh\cdot\dl\mB_j \dE +
\int_{\partial E} \vvh\cdot(\Gr\mB_j\,\nn)\de\nonumber\\
&=& - \int_E \vvh\cdot\dl\mB_j \dE +
\sum_{e\in\partial E}\int_e \vvh\cdot(\Gr\mB_j\,\nn_e)\de\,.
\nonumber\\
\label{eqn:firstPart}
\end{eqnarray}
As we already said in the introduction of this section,
the virtual function $\vvh$ is a known vectorial polynomial on the boundary so 
integrals over edges are computable, in particular we integrate a polynomial of degree $2k-1$ on $e$ (see Reamrk \ref{rem:boundary1}).
Let us focus on the term inside the polygon $E$.
We observe that $\dl\mB_j$ is a vectorial monomial 
then the following relation holds
$$
\dl\mB_j = 
\frac{a_1}{h_E^2}\,\mVect{\aalpha_1}{\emptyset} + 
\frac{a_2}{h_E^2}\,\mVect{\aalpha_2}{\emptyset} + 
\frac{a_3}{h_E^2}\,\mVect{\emptyset}{\aalpha_3} + 
\frac{a_4}{h_E^2}\,\mVect{\emptyset}{\aalpha_4}\,.
$$
Finding the coefficients $a_s$ and the multi-indexes $\aalpha_s$ is easy 
since we are dealing with scaled monomials.
We substitute such expression in the first integral of Equation~\eqref{eqn:firstPart}
and we get
\begin{equation}
\label{eqn:stepWithABPN}
\int_E \vvh\cdot\left(\frac{a_1}{h_E^2}\,\mVect{\aalpha_1}{\emptyset} + 
\frac{a_2}{h_E^2}\,\mVect{\aalpha_2}{\emptyset} + 
\frac{a_3}{h_E^2}\,\mVect{\emptyset}{\aalpha_3} + 
\frac{a_4}{h_E^2}\,\mVect{\emptyset}{\aalpha_4}\right)\dE\,.
\end{equation}
Let us consider one of these integrals, 
similar considerations can be done for the other terms.
We exploit Proposition~\ref{prop:decompMono} and we integrate by parts
\begin{eqnarray}
 \int_E \vvh\cdot\mVect{\aalpha_1}{\emptyset}\dE &=& 
\int_E \vvh\cdot
(b_1\nabla m_{\bbeta_1} + 
b_2\,\mPerp \, m_{\bbeta_2})\dE \nonumber\\ &=&
b_1\int_E \vvh\cdot\nabla m_{\bbeta_1}\dE
+ b_2\int_E \vvh\cdot \mPerp m_{\bbeta_2}\dE\,,\nonumber\\
\label{eqn:step1}
\end{eqnarray}
where we have defined the scalar coefficients $b_s$ and 
the multi-indexes $\bbeta_s$ from Proposition~\ref{prop:decompMono}.
According to such proposition $-1 \leq |\bbeta_2| \leq k-3$ so 
the last integral is the degree of freedom $\mathbf{D}^{\mathbf{m}^{\perp}}$
and we can compute it.
The first integral requires additional steps:
\begin{eqnarray}
\int_E \vvh\cdot\nabla m_{\bbeta_1}\dE &=& 
- \int_E (\dd \,\vvh)\,m_{\bbeta_1} \dE
+  \int_{\partial E} (\vvh\cdot\nn)\,m_{\bbeta_1}\de\nonumber\\ &=&
- \int_E (\dd \, \vvh)\,m_{\bbeta_1} \dE
+  \sum_{e\in\partial E}\int_{e} (\vvh\cdot\nn_e)\,m_{\bbeta_1}\de\,.\nonumber\\
\label{eqn:step2}
\end{eqnarray}
The integral inside the element is $\mathbf{D^{\rm div}}$ degree of freedom
since $1 \leq |\bbeta_1| \leq k-1$.
The integral over the boundary is computable 
since the virtual function $\vvh$ is a vectorial polynomial on each edge $e$ (we are integrating a polynomial of degree $2k-1$, that is computable by Remark \ref{rem:boundary1}).

The explicit computation of the local stiffness matrix ${\bm K}^{\Gr, E}_h$ (cf. \eqref{eq:aGr}) now follows the guidelines given in \cite{autostoppisti}.

\subsection{$\grepseps$--projection $\Peps$}\label{sub:piEps}

In this subsection we consider the $\grepseps$-projection, $\Peps$ (cf. Definition \eqref{eq:Peps_k^E}).
We proceed in a similar way to the $\PN$.
Given a virtual function $\vvh\in\VV_h^E$, 
its projection is a vector valued polynomial of degree $k$.
We write it in terms of the monomial basis $[\mathbb{M}_k(E)]^2$, 
(cf. \eqref{eqn:PiNablaSviluppo}, considering the projection $\Peps$), 
and we also re-write Definition~\eqref{eq:Peps_k^E} in terms of scaled monomials, obtaining
\begin{equation}
\left\{
\begin{array}{rlll}
\mathlarger{\int_E} \epseps (\vvh - \Peps \vvh) \, : \, \epseps(\mB_j)\dE &=& 0 &
\text{for all $\mB_j\in [\mathbb{M}_k(E)]^2 \setminus \mathbb{K}^{\epseps}(E)$,}
\\
\mathlarger{\int_{\partial E}} (\vvh - \Peps \vvh)\cdot\mB_j \de &=& 0 &
\text{for all $\mB_j\in  \mathbb{K}^{\epseps}(E)$,}
\end{array}
\right.
\label{eqn:piEpsilonMono}
\end{equation}
$$
\mathbb{K}^{\epseps}(E) := 
\left\{
\left(\begin{array}{c} 1 \\ 0 \end{array}\right)\,,
\left(\begin{array}{c} 0 \\ 1 \end{array}\right)\,,
\mPerp
\right\}\,.
$$
It is easy to prove that 
Equation~\eqref{eqn:piEpsilonMono} defines a set of linearly independent conditions 
which uniquely determine the unknown coefficients $\zeta_i$.
Also in this case, Equation~\eqref{eqn:piEpsilonMono} is in the algebraic form detailed in Remark \ref{rem:system}.
Consequently, to find the projection we have to directly solve a linear system.

\begin{rem}
The set $[\mathbb{M}_k(E)]^2 \setminus \mathbb{K}^{\epseps}(E)$ contains the set of scaled monomials which does not belong to the kernel of the operator $\epseps$  (the so-called rigid body motions)
so that such conditions do not become trivial (0=0).
\end{rem}

\noindent 
We analyse only the first conditions in Equation~\eqref{eqn:piEpsilonMono}, for the boundary ones we recall that the virtual functions are explicitly known on $\partial E$ and the boundary integrals are exactly computable in the sense of Remark \ref{rem:boundary1}. 
Therefore we have the linear system 
$$
\sum_{i=0}^{2\pi_k} \zeta_i\,\int_E \epsilon(\mB_i)\, : \, \epsilon(\mB_j)\dE  = 
\int_E \epsilon(\vvh) \, : \, \epsilon(\mB_j)\dE\,.
$$
The left-hand side is computable since it involves only vectorial scaled monomials.
We have to verify if the right-hand side 
is computable from the DoFs values of the virtual function~$\vvh$.
Simple integration by parts and yields 
\begin{eqnarray*}
\int_E \epseps(\vvh) \, : \, \epseps(\mB_j)\dE &=& 
- \int_E \vvh \cdot \ddVect(\epseps(\mB_j))\dE + 
\int_{\partial E} \vvh \cdot (\epseps(\mB_j)\,\nn)\de \\ &=&
- \int_E \vvh \cdot \ddVect(\epseps(\mB_j))\dE + 
\sum_{e\in \partial E} \int_e \vvh \cdot (\epseps(\mB_j)\,\nn_e)\de 
\end{eqnarray*}
As usual, the boundary term is computable (we integrate on each edge $e$ a polynomial of degree $2k-1$).
Concerning the element integral, notice that $\ddVect(\epseps(\mB_j))$  can be written as 
$$
\ddVect(\epseps(\mB_j)) = 
\frac{a_1}{h_E^2}\,\mVect{\aalpha_1}{\emptyset} + 
\frac{a_2}{h_E^2}\,\mVect{\aalpha_2}{\emptyset} + 
\frac{a_3}{h_E^2}\,\mVect{\emptyset}{\aalpha_3} + 
\frac{a_4}{h_E^2}\,\mVect{\emptyset}{\aalpha_4}\,.
$$
Now we can proceed as before, see Equations~\eqref{eqn:step1} and~\eqref{eqn:step2}.
Again the local stiffness matrix ${\bm K}^{\grepseps, E}_h$ (cf. \eqref{eq:aeps}) is computed following in a rather slavish way the reference \cite{autostoppisti}.

\subsection{$L^2$--projection $\PiZero$ projection}\label{sub:piZero}

In this subsection we verify the computability of the $L^2$-projection operator $\PiZero$ (cf. Definition \eqref{eq:P0_k^E}). 
In particular we will exploit the so-called enhanced property of the virtual space \eqref{eq:V_h^E}.
As we have done for the previous polynomial projections, 
we look for the unknown coefficients from the relations
in Definition~\eqref{eq:P0_k^E} written in terms of vectorial scaled monomials
\begin{equation}
\int_E (\vvh - \PiZero \vvh) \cdot \mB_j \dE = 0\qquad\text{for all $\mB_j\in[\mathbb{M}_k(E)]^2$.} \\
\label{eqn:piZeroMono}
\end{equation}
It is easy to show that such conditions are sufficient to find the unknown polynomial coefficients $\zeta_i$ (cf. \eqref{eqn:PiNablaSviluppo} with respect to $\PiZero$).
We proceed as before and we put in Equation~\eqref{eqn:piNablaMono} the polynomial $\PiZero \vvh$ 
written in terms of vectorial scaled monomials, i.e. 
$$
\sum_{i=1}^{2\pi_k} \zeta_{i}\,\int_E \mB_{i}\cdot\mB_{j}\dE  = \int_E \vvh\cdot\mB_{j}\dE\,.
$$
The left-hand side is computable, while the right-hand side involves the virtual function $\vvh$
so we have to verify if it is computable via the degrees of freedom of $\vvh$.
Fist we exploit Proposition~\ref{prop:decompMono} and we get 
\begin{eqnarray*}
\int_E \vvh\cdot\mB_j\dE &=& \int_E \vvh\cdot\left(b_1\,\nabla\m{\bbeta_1} 
+ b_2\,\mPerp\m{\bbeta_2}\right)\dE\nonumber\\
&=& b_1\underbrace{\int_E \vvh\cdot\nabla\m{\bbeta_1}\dE}_{(a)} +\,
b_2\underbrace{\int_E \vvh\cdot\mPerp\,\m{\bbeta_2}\dE}_{(b)}\,,\nonumber\\
\end{eqnarray*}

\noindent Let us consider these two terms separately.
\begin{enumerate}[(a)]
 \item\label{enum:caseA} In the first one we integrate by parts and we get 
\begin{eqnarray*}
\int_E \vvh\cdot\nabla\m{\bbeta_1}\dE &=& -\int_E (\dd \, \vvh)\,\m{\bbeta_1}\dE + 
\int_{\partial E} (\vvh\cdot\nn)\,\m{\bbeta_1}\de \\ &=&
-\int_E (\dd \,\vvh)\,\m{\bbeta_1}\dE + 
\sum_{e\in\partial E}\int_e (\vvh\cdot\nn_e)\,\m{\bbeta_1}\de\,.
\end{eqnarray*}
Such integrals are computable. 
The way of computing the fist integral depends on the degree of $\m{\bbeta_1}$.
If $|\bbeta_1|\leq k-1$, it is a degrees of freedom, $\mathbf{D^{\rm div}}$.
In all the other cases, $k \leq |\bbeta_1|\leq k+1$,
we already prove that we can find the exact expression of $\dd \, \vvh$, see Subsection~\ref{sub:div},
so, since $\dd \, \vvh \in\mathbb{P}_{k-1}(E)$, we compute such integral exactly.
Regarding the boundary term, we recall that  $\vvh$ is a known vectorial polynomial on each edge $e$. However, contrary to the previous cases, we are integrating a polynomial of degree $2k+1$, therefore the $(k+1)$ Gauss--Lobatto values (i.e. the DoFs $\mathbf{D^e}$) are not sufficient to compute exactly this integral.
In such case we need to reconstruct the polynomial  $\vvh$ on each edge $e$ and  then employ a quadrature rule of degree $2k+1$.

\item\label{enum:caseB} Also this integral depends on the degree of the monomial $\m{\bbeta_2}$.
Indeed, if $|\bbeta_2|\leq k-3$, it is a degrees of freedom $\mathbf{D}^{\mathbf{m}^{\perp}}$.
Otherwise, when $k-2 \leq |\bbeta_2|\leq k-1$,
we have to exploit the enhancing condition and 
compute such integral via the projection operator $\PN$, i.e.
$$
\int_E \vvh\cdot\mPerp\,\m{\bbeta_2}\dE = \int_E \PN \vvh\cdot\mPerp\,\m{\bbeta_2}\dE\qquad
k-2 \leq|\bbeta_2|\leq k-1\,.
$$
\end{enumerate}
Now the  the local matrix ${\bm K}^{0, E}_h$ (cf. \eqref{eq:a0}) is built using the guide \cite{autostoppisti}.

\subsection{$\PiGrad \Gr$ projection}\label{sub:piGrad}

In this subsection we verify the computability of the $L^2$-projection operator $\PiGrad \colon \Gr(\VV_h^E) \to [\Pk_{k-1}(E)]^{2 \times 2}$. 
We consider the basis $[\mathbb{M}_{k-1}(E)]^{2\times 2}$ and 
we also write the projection $\PiGrad \Gr \vvh$ in terms of such basis functions:
$$
\PiGrad \Gr \vvh = \sum_{i=1}^{4 \pi_{k-1}} \zeta_i \, \mMatB_{i} \,.
$$
Then, starting from the matrix counterpart of conditions~\eqref{eq:P0_k^E},
we get these set of equations 
\begin{equation}
\sum_{i=1}^{4\pi_{k-1}} \zeta_i \int_E \mMatB_{i}\,:\,\mMatB_{j}\dE = 
\int_E \Gr\vvh\,:\,\mMatB_{j}\dE\quad \text{for  $j=1, \, \dots \,,4\pi_{k-1}$.}
\label{eqn:conForGradProj}
\end{equation}
As for all the other projection operators,
Equation~\eqref{eqn:conForGradProj} defines a linear system 
whose unknowns are the coefficients $\zeta_i$ of $\PiGrad \nabla \vvh$. 

The left-hand side of Equation~\eqref{eqn:conForGradProj} is computable 
since it involves only matrix scaled monomials.
Regarding the right-hand side we proceed as for the other cases.
Integration by parts yields
\begin{eqnarray*}
\int_E \Gr\vvh\,:\,\mMatB_j\dE &=& -\int_E \vvh\cdot (\ddVect \, \mMatB_j)\dE +
\int_{\partial E} \vvh\cdot(\mMatB_j\,\nn)\de \\ &=& 
-\int_E \vvh\cdot(\ddVect \, \mMatB_j)\dE + 
\sum_{e\in\partial E} \int_{e} \vvh\cdot(\mMatB_j\,\nn_e)\de\,.
\end{eqnarray*}
The integral over the edges $e$ is computable since the virtual function $\vvh$ is a vectorial polynomial on the edges (in particular we ingrate along the edge $e$ a polynomial of degree $2k-1$, see Remark \ref{rem:boundary1}).
To show the computability of the internal integral,
we observe that 
$$
\ddVect \,\mMatB_j = 
\frac{a_1}{h_E}\,\mVect{\aalpha_1}{\emptyset} + 
\frac{a_2}{h_E}\,\mVect{\aalpha_2}{\emptyset} + 
\frac{a_3}{h_E}\,\mVect{\emptyset}{\aalpha_3} + 
\frac{a_4}{h_E}\,\mVect{\emptyset}{\aalpha_4}\,,
$$
where, as before, the coefficients $a_s$ and the multi-indexes $\aalpha_s$ can be easily found since we are dealing with scaled monomials.
Consequently, the computability of $\PiGrad$ follows from same arguments of $\PN$, 
see Equation~\eqref{eqn:stepWithABPN}.

\begin{rem}
The computation of $\PP0\epseps(\vv)$ follows the same strategy of $\PP0\Gr\vv$ 
so we omit the explicit construction of such projection operator.
\label{rem:piEpsGrad}
\end{rem}

\section{Numerical Results}
\label{sec:tests}

We present three numerical experiments to exploit the behaviour of the proposed virtual elements family 
for the mixed problems such as Stokes, Darcy, Brinkman and Navier--Stokes equations.
More specifically we assess the actual performance of the virtual element method for high-order polynomial degrees (up to $k=6$).

Given $\ff \in [L^2(\Omega)]^2$ and $g \in L^2(\Omega)$, and
referring to \eqref{eq:V_h}, \eqref{eq:Q_h}, \eqref{eq:ahglobal}, \eqref{eq:bform}, \eqref{eq:fh},
the virtual elements approximation of the general mixed problem has the form  
\begin{equation}
\label{eq:virtual mixed}
\left\{
\begin{aligned}
& \text{find $(\uu_h,\, p_h) \in \VV_h \times Q_h$ s.t.}
\\
& a_h(\uu_h, \, \vv_h) + b(\vv_h, \, p_h) = (\ff_h, \, \vv_h) 
&\qquad &\text{for all $\vv_h \in \VV_h$,}
\\
& b(\uu_h, \, q_h) = (g, \, q_h) 
&\qquad &\text{for all $q_h \in Q_h$.}
\end{aligned}
\right.
\end{equation}
We now explore the algebraic structure of the mixed problem above.
Consider a generic polygonal mesh $\Omega_h := \{ E_i\}_{i=1}^{n_P}$ of $\Omega$.
We denote with ${\bm K}_h$ 
the global counterpart of
one of the local stiffness matrices in \eqref{eq:a0}, \eqref{eq:aGr} and \eqref{eq:aeps},
and with ${\bm B}$ (cf. \eqref{eq:bmatrix}), ${\bm f}_h$, ${\bm g}$ (cf. \eqref{eq:FG_h}) the global version of the div-matrix and discrete right-hand side respectively. 
Denoting with $\texttt{GNDoF}$ the total amount of global velocity DoFs,
the velocity and pressure solutions can be expressed in terms of the global bases
\[
\uu_h = \sum_{i=1}^{\texttt{GNDoF}} \chi_i \, \fg_i 
\qquad
\text{and}
\qquad
p_h = \sum_{j=1}^{n_P} \frac{h_{E_j}}{|E_j|} \sum_{{\ell}=1}^{\pi_{k-1}} \rho_{\ell}^j \, m_{\ell}
\]
and the vectors
\[
{\bm \chi} := \left( \chi_i \right)_i 
\qquad 
\text{and}
\qquad
{\bm \rho} := \left[\begin{array}{@{}c|c@{}|c@{}}
{\bm \rho}^1 & \dots & {\bm \rho}^{n_P} 
\end{array} \right]
\]
are the unknowns of the associated discrete problem.
The discretization of the mixed problem results in a linear algebraic system of the form 
\begin{equation}
\label{eq:algebric}
\left[\begin{array}{@{}c|c@{}|c@{}}
  \begin{matrix}
   &  &  
   \\
   & {\bm K}_h  &
   \\
   &  &
  \end{matrix}
  &
  {\bm 0} \, \,
  &
  \begin{matrix}
   &  &  
   \\
   & {\bm B}^{\rm T}  &
   \\
   &  &
  \end{matrix}
  \\
  \hline
  &&\\[-1em]
  {\bm 0}
  &
  0 \, \,
  &
  {\bm \sigma}^{\rm T}
  \\
  \hline
  \begin{matrix}
   &  &  
   \\
   & {\bm B}  &
   \\
   &  &
  \end{matrix}
  &
  {\bm \sigma} \, \,
  &
  \begin{matrix}
   &  &  
   \\
   & {\bm 0}  &
   \\
   &  &
  \end{matrix}
\end{array} \right]
\, \,
\left[\begin{array}{@{}c}
  \begin{matrix}
  \\
  \, \, {\bm \chi} \\
  \\ 
  \end{matrix}	
  \\
\hline
  \\[-1em]
  \, \, \, \lambda
  \\
\hline  
  \begin{matrix}
  \\
  \, \, {\bm \rho} \\
  \\
  \end{matrix} 
\end{array}\right]
\, = 
\,
\left[\begin{array}{@{}c}
  \begin{matrix}
  \\
  \, \, {\bm f}_h \\
  \\ 
  \end{matrix}	
  \\
\hline
  \\[-1em]
  \, \, 0
  \\
\hline  
  \begin{matrix}
  \\
  \, \,{\bm g} \\
  \\
  \end{matrix} 
\end{array}\right]
\end{equation}
where
\[
{\bm \sigma} := 
\left[\begin{array}{@{}c|c@{}|c@{}}
{\bm \sigma}^1 & \dots & {\bm \sigma}^{n_P} 
\end{array} \right]
\]
with
\[
\left({\bm \sigma}^j \right)_{\ell} = \frac{h_{E_j}}{|E_j|} \int_{E_j} m_{\ell} \, {\rm d}E_j
\qquad
\text{for $j=1, \, \dots\, , n_P$ and ${\ell}=1, \, \dots \, , \pi_{k-1}$.} 
\]
Note that the second block-row in \eqref{eq:algebric} represents the ``zero averaged pressure'' constraint and $\lambda$  is the associated Lagrange multiplier.

Concerning the order of accuracy of the method, 
let $(\uu,\,p)$ and $(\uuh,\,\p_h)$ be
the continuous and its corresponding discrete solution given by \eqref{eq:virtual mixed}.
Then, the expected rate of convergence for the errors are 
\begin{itemize}
\item \textbf{$H^1$--velocity error}:
\[
\|\Gr \uu - \Gr\uuh\|_{L^2(\Omega)} \lesssim h^k \, |\uu|_{k+1} + h^{k+2} \, |\ff|_{k+1} \,,   
\]
\item \textbf{$L^2$--velocity error}:
\[
\| \uu - \uuh\|_{L^2(\Omega)} \lesssim h^{k+1} \, |\uu|_{k+1} + h^{k+3} \, |\ff|_{k+1} \,,
\]
\item \textbf{$L^2$--pressure error}:
\[
\| \p - \ph\|_{L^2(\Omega)} \lesssim h^{k} \, |\uu|_{k+1}  + h^k \, |\p|_{k} + h^{k+2} \, |\ff|_{k+1}  \,. 
\]
\end{itemize}
Since the VEM velocity solution $\uuh$ is not explicitly known point-wise inside the elements,
in order to check the actual performance of the method,
we compute the errors comparing $\uu$ with a suitable polynomial projection of the discrete solution
$\uuh$.
Note that the pressure variable $\ph$ is a piecewise polynomial so 
we can use it directly to compute the pressure error.
Therefore we consider the following computable error quantities:
\begin{align*}
& \text{\texttt{error}}(\uu, H^1) := \sqrt{\sum_{E\in\Omega_h}\|\Gr \uu- \Gr \PN\uuh\|_{L^2(E)}^2}\,,
 \\
& \text{\texttt{error}}(\uu, L^2) := \sqrt{\sum_{E\in\Omega_h}\|\uu-\PiZero\uuh\|_{L^2(E)}^2}\,,
 \\
& \text{\texttt{error}}(\p, L^2) := \|\p-\ph\|_{L^2(\Omega)}\,.
\end{align*}


%

In the experiments we
consider three sequences of finer meshes of the unit square $[0,\,1]^2$,
see Figure~\ref{fig:meshes}:
\begin{itemize}
 \item \texttt{quad}, a mesh composed by structured squares;
 \item \texttt{hexa}, a mesh composed by distorted hexagons;
 \item \texttt{voro}, a Voronoi tessellation composed by general polygons.
\end{itemize}
We associate with each discretization a mesh-size
$$
h:= \frac{1}{n_P}\sum_{i=1}^{n_P} h_{E_i}\,.
$$

\begin{figure}[!htb]
\begin{center}
\begin{tabular}{ccc}
\includegraphics[width=0.3\textwidth]{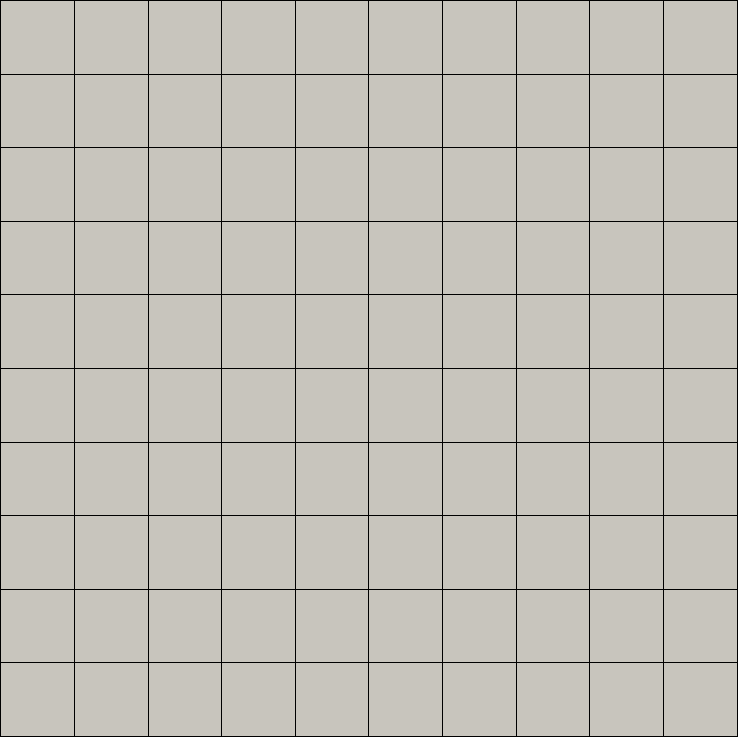}&
\includegraphics[width=0.3\textwidth]{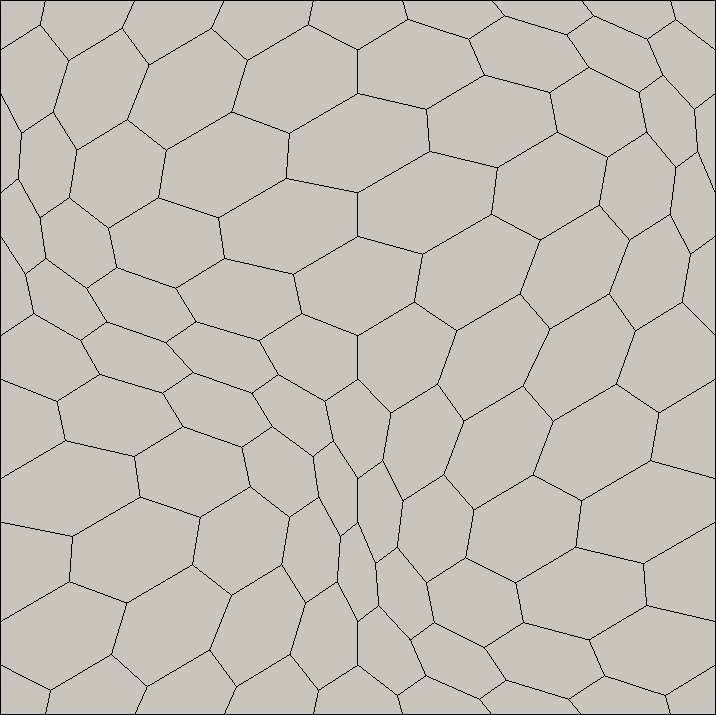}&
\includegraphics[width=0.3\textwidth]{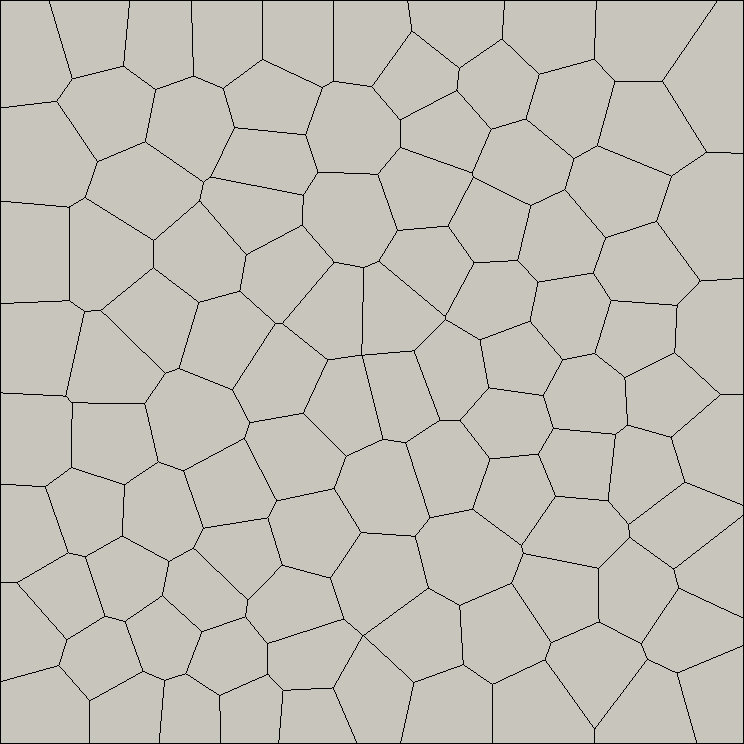}\\
\texttt{quad}&
\texttt{voro}&
\texttt{hexa}\\
\end{tabular}
\end{center}
\caption{An example of one of the meshes used to make the convergence analysis}
\label{fig:meshes}
\end{figure}

\subsection{Stokes Equations}\label{sub:stokes}

In this test we solve the virtual Stokes problem (in the ``epsilon form'')
\begin{equation}
\left\{
\begin{aligned}
& \text{find $(\uu_h,\, p_h) \in \VV_h \times Q_h$ s.t.}
\\
& a_h^{\grepseps}(\uu_h, \, \vv_h) + b(\vv_h, \, p_h) = (\ff_h, \, \vv_h) 
&\qquad &\text{for all $\vv_h \in \VV_h$,}
\\
& b(\uu_h, \, q_h) = 0
&\qquad &\text{for all $q_h \in Q_h$.}
\end{aligned}
\right.
\label{eqn:stokesProb}
\end{equation}
where the bilinear operators $a^{\grepseps}_h(\cdot,\cdot)$ and $b(\cdot,\cdot)$ are 
defined in Equations~\eqref{eq:aheps} and~\eqref{eq:bform}, respectively,
while the discrete load is defined in \eqref{eq:fh}.
The load term $\ff$ and the boundary conditions of the continuous formulation of 
Problem~\eqref{eqn:stokesProb} are chosen in such a way that the exact solution is 
$$
\uu(x,\,y) = \left(
\begin{array}{r}
 0.5\,\sin(2\pi x)\,\sin(2\pi x)\,\sin(2\pi y)\,\cos(2\pi y)\\
-0.5\,\sin(2\pi y)\,\sin(2\pi y)\,\sin(2\pi x)\,\cos(2\pi x)
\end{array}
\right) \,,
\qquad 
\p(x,\,y) = \sin(2\pi x)\,\cos(2\pi y)\,.
$$

In Figures~\ref{fig:quadStokes},~\ref{fig:hexaStokes} and~\ref{fig:voroStokes}
we show the convergence lines with different VEM approximation degrees (up to $k=6$) for the sequences of meshes listed above.
The trend of the error is the expected one and 
it is astonishingly stable with high approximation degrees $k$.
Indeed, the error $\text{\texttt{error}}(\uu, H^1)$ reaches small values, 
close to the machine precision, for $k=6$ and the finest mesh.
Here we do not show the results for $\text{\texttt{error}}(\uu, L^2)$ since they have the expected trend too.

\newcommand\graphsize{0.48}

\begin{figure}[!htb]
\begin{center}
\begin{tabular}{cc}
\includegraphics[width=\graphsize\textwidth]{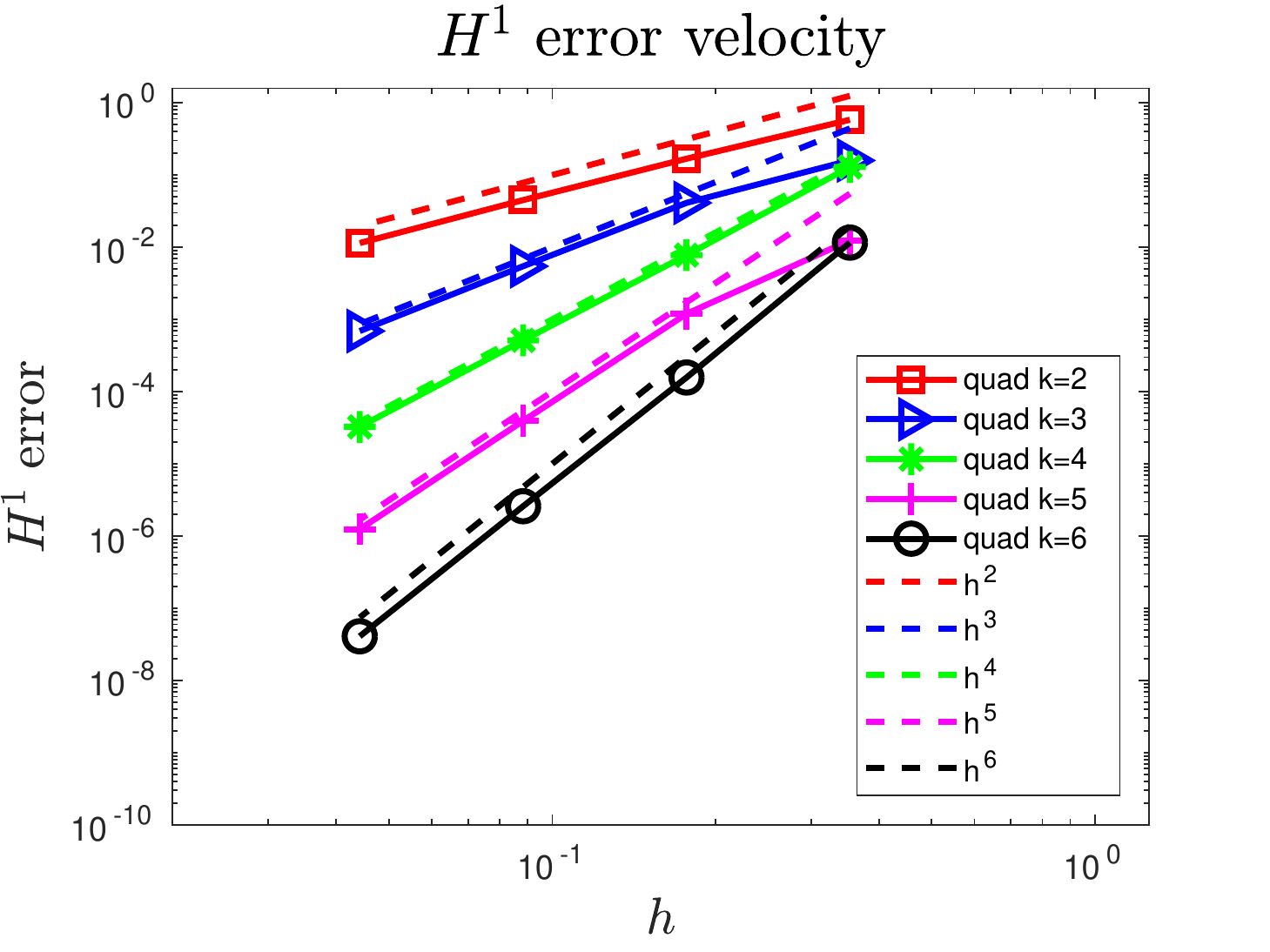} &
\includegraphics[width=\graphsize\textwidth]{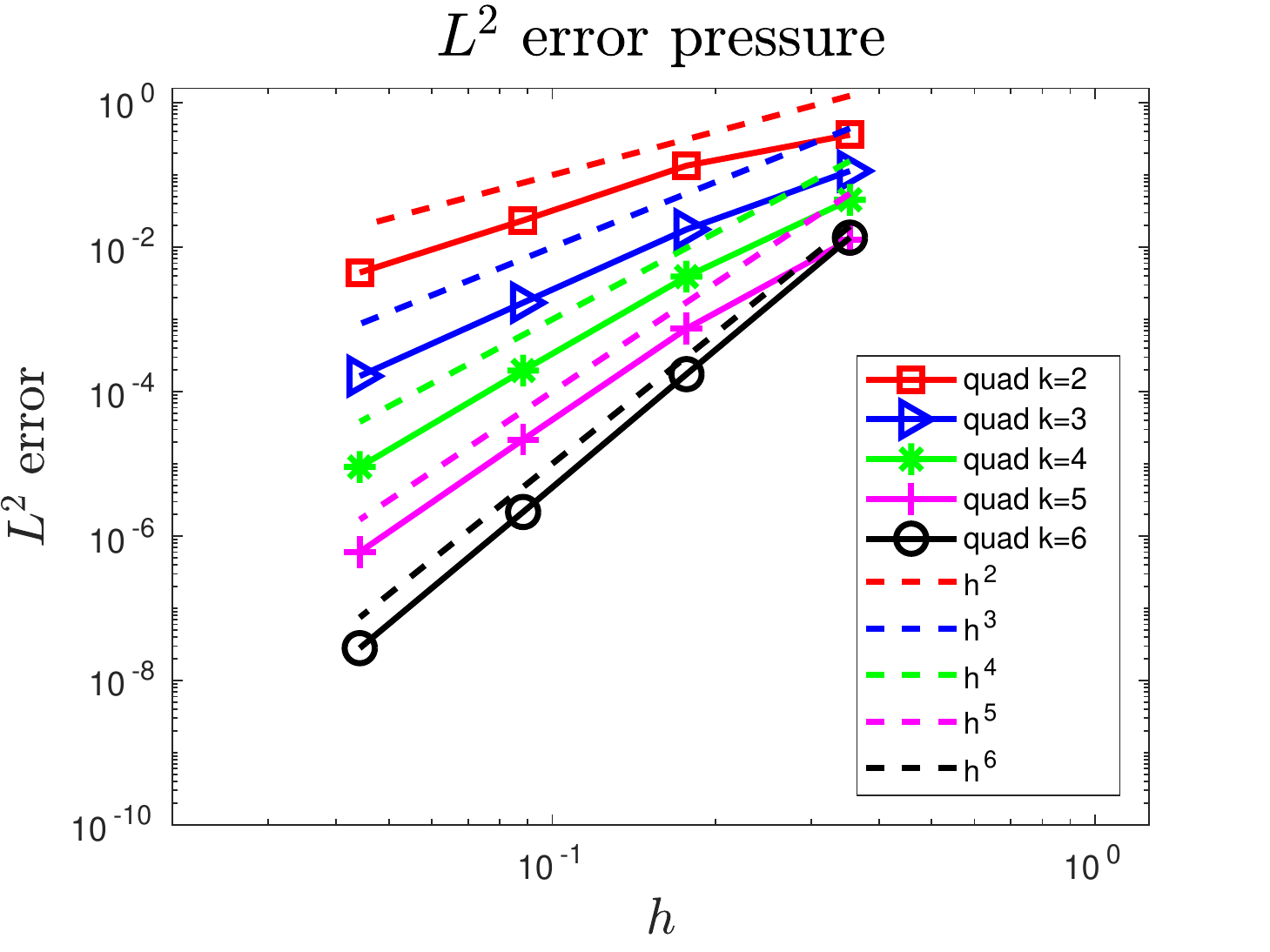} \\
\end{tabular}
\end{center}
\caption{Stokes Problem. Convergence lines for the \texttt{quad} meshes with $k=2,3,4,5$ and 6.}
\label{fig:quadStokes}
\end{figure}

\begin{figure}[!htb]
\begin{center}
\begin{tabular}{cc}
\includegraphics[width=\graphsize\textwidth]{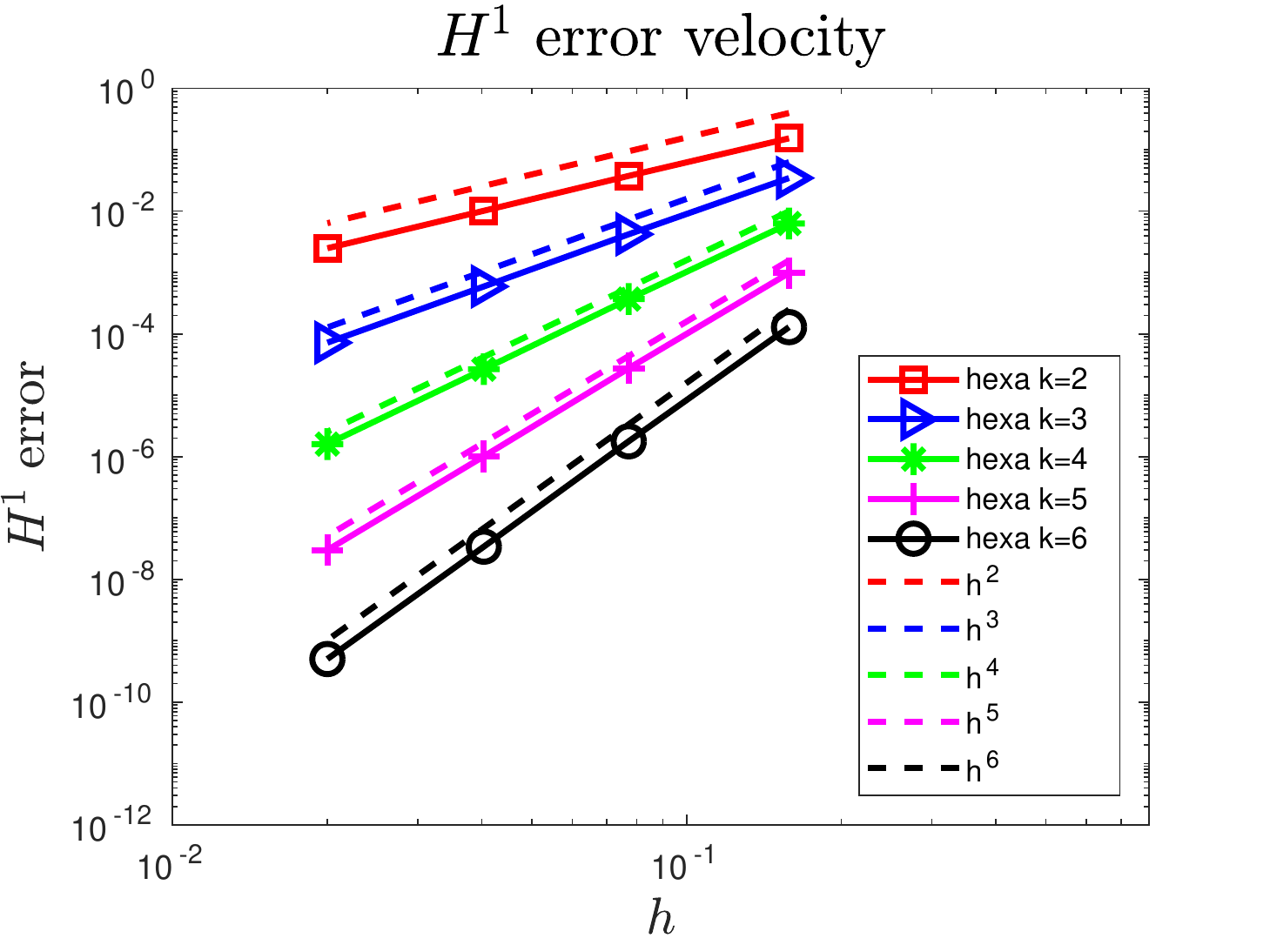} &
\includegraphics[width=\graphsize\textwidth]{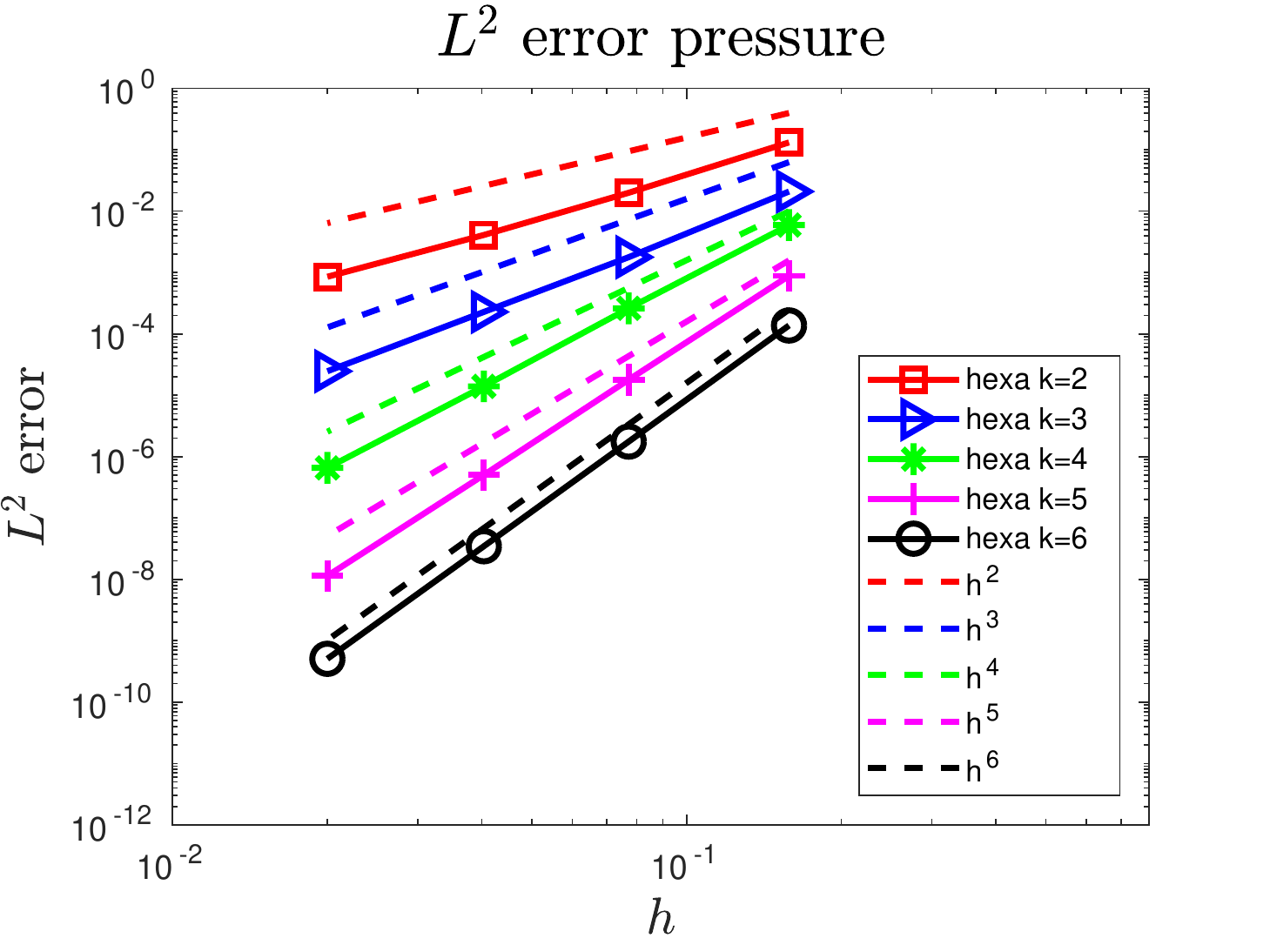} \\
\end{tabular}
\end{center}
\caption{Stokes Problem. Convergence lines for the \texttt{hexa} meshes with $k=2,3,4,5$ and $6$.}
\label{fig:hexaStokes}
\end{figure}

\begin{figure}[!htb]
\begin{center}
\begin{tabular}{cc}
\includegraphics[width=\graphsize\textwidth]{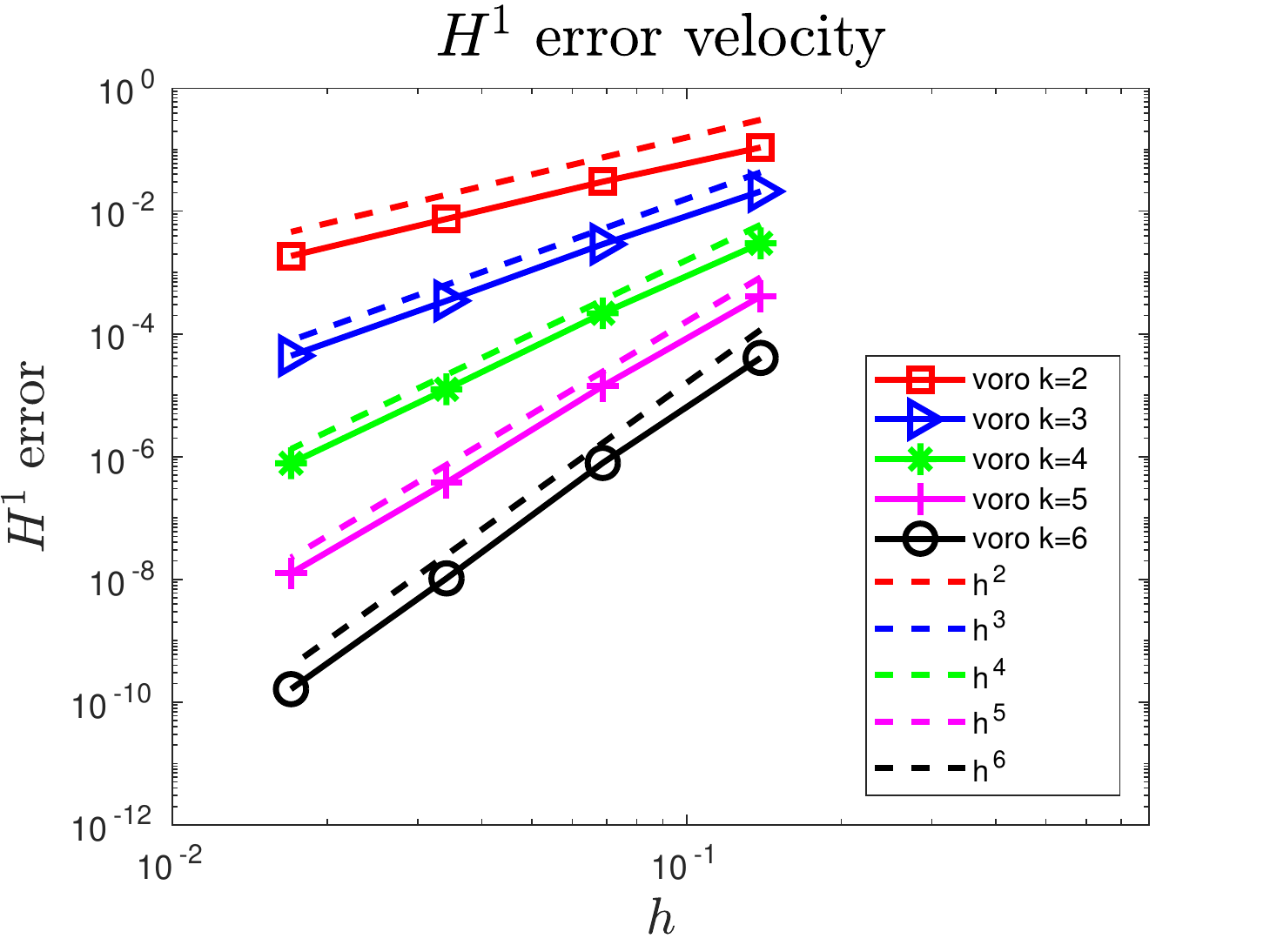} &
\includegraphics[width=\graphsize\textwidth]{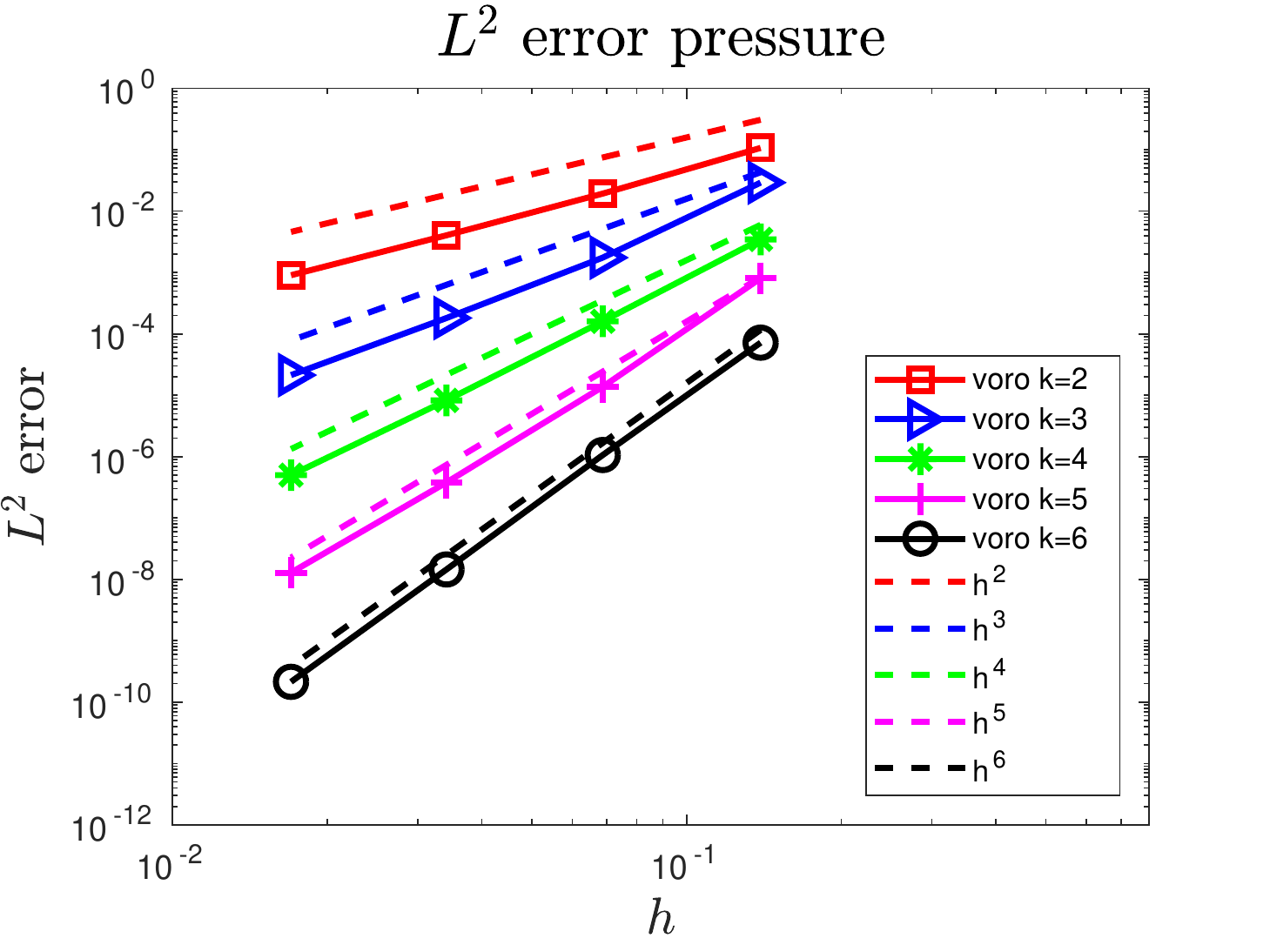} \\
\end{tabular}
\end{center}
\caption{Stokes Problem. Convergence lines for the \texttt{voro} meshes with $k=2,3,4,5$ and $6$.}
\label{fig:voroStokes}
\end{figure}

\subsection{Darcy Equations}\label{sub:darcy}

In this subsection we are solving the Darcy problem
\begin{equation}
\left\{
\begin{aligned}
& \text{find $(\uu_h,\, p_h) \in \VV_h \times Q_h$ s.t.}
\\
& a_h^0(\uu_h, \, \vv_h) + b(\vv_h, \, p_h) = 0 
&\qquad &\text{for all $\vv_h \in \VV_h$,}
\\
& b(\uu_h, \, q_h) = (g, \, q_h) 
&\qquad &\text{for all $q_h \in Q_h$.}
\end{aligned}
\right.
\label{eqn:darcyProb}
\end{equation}
where the bilinear operators $a^{0}_h(\cdot,\cdot)$ and $b(\cdot,\cdot)$ are 
defined in Equations~\eqref{eq:ah0} and~\eqref{eq:bform}, respectively.
We the continuous version of \eqref{eqn:darcyProb} whose exact solution is the pair 
$$
\uu(x,\,y) = \left(
\begin{array}{r}
-\pi \sin(\pi x)\,\cos(\pi y)\\
-\pi \cos(\pi x)\,\sin(\pi y)
\end{array}
\right) \,,
\qquad
\p(x,\,y) = \cos(\pi x)\,\cos(\pi y)\,.
$$

In Figures~\ref{fig:quadDarcy},~\ref{fig:hexaDarcy} and~\ref{fig:voroDarcy} we provide the convergence lines.
The trend of the error is the expected one for the cases $k=2,3,4$ and $5$.
However, in the last step of the convergence lines for $k=6$ associate 
with the  $\text{\texttt{error}}(\uu, L^2)$ the lines does not follow the theoretical trend.
This fact is probably due to the matrix conditioning.
Indeed, we are solving a large linear system of high degree and the error is close to the machine precision.
We get the expected trend also for $\text{\texttt{error}}(\uu, H^1)$, 
but, as for $\text{\texttt{error}}(\uu, L^2)$, in the last steps of $k=6$ it has a plateau around $10^{-12}$.

\begin{figure}[!htb]
\begin{center}
\begin{tabular}{cc}
\includegraphics[width=\graphsize\textwidth]{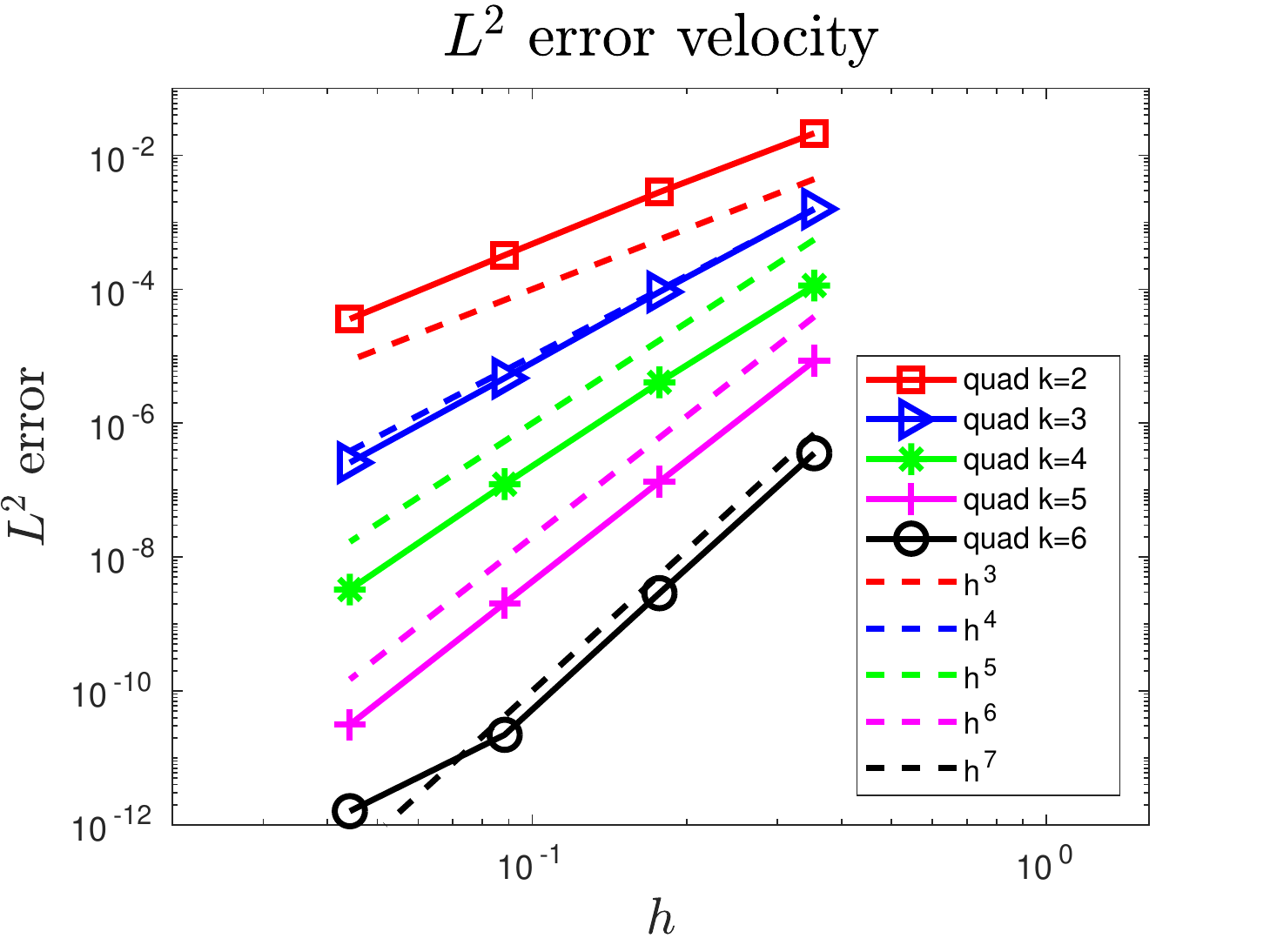} &
\includegraphics[width=\graphsize\textwidth]{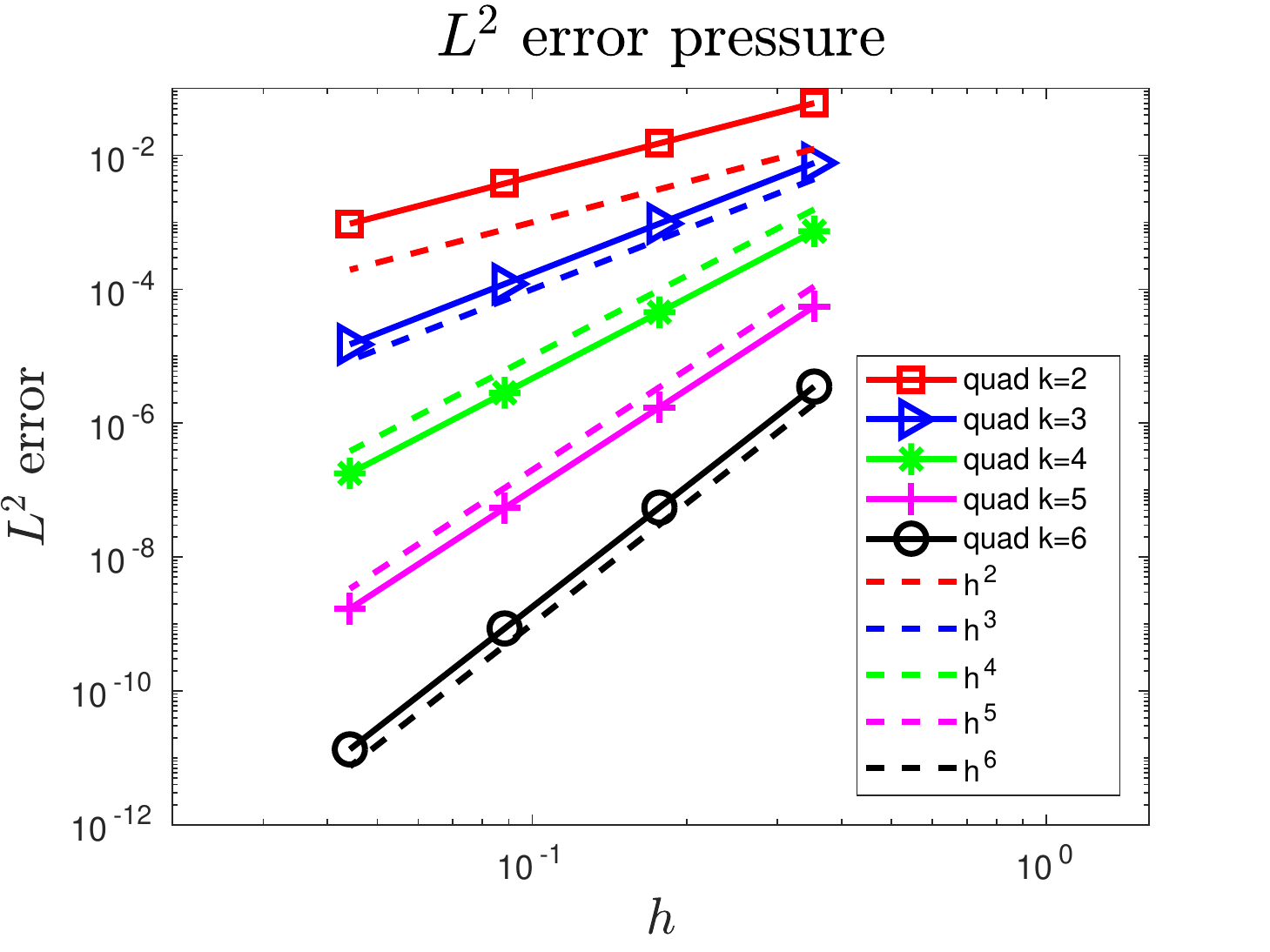} \\
\end{tabular}
\end{center}
\caption{Darcy Problem. Convergence lines for the \texttt{quad} meshes with $k=2,3,4,5$ and 6.}
\label{fig:quadDarcy}
\end{figure}

\begin{figure}[!htb]
\begin{center}
\begin{tabular}{cc}
\includegraphics[width=\graphsize\textwidth]{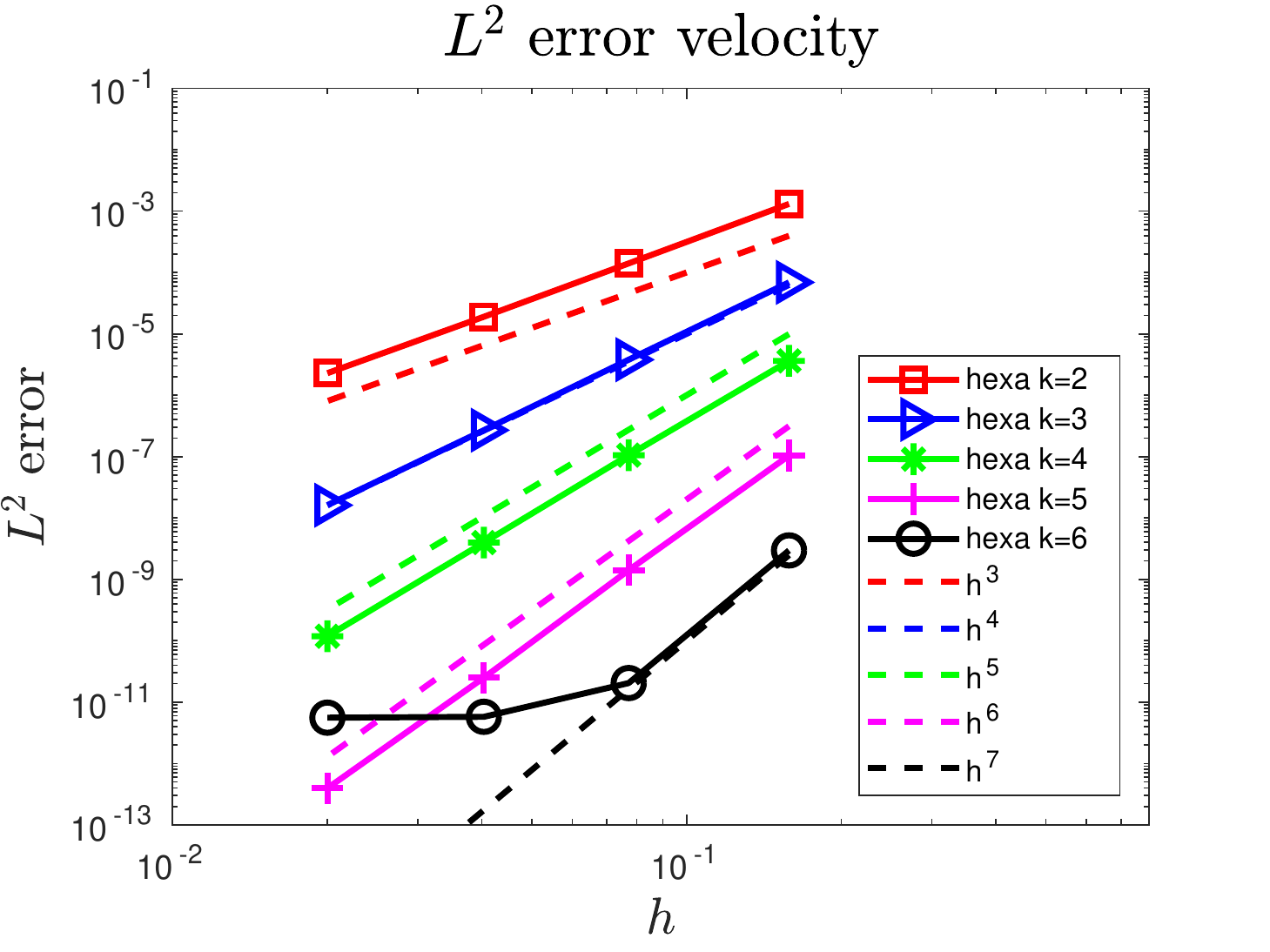} &
\includegraphics[width=\graphsize\textwidth]{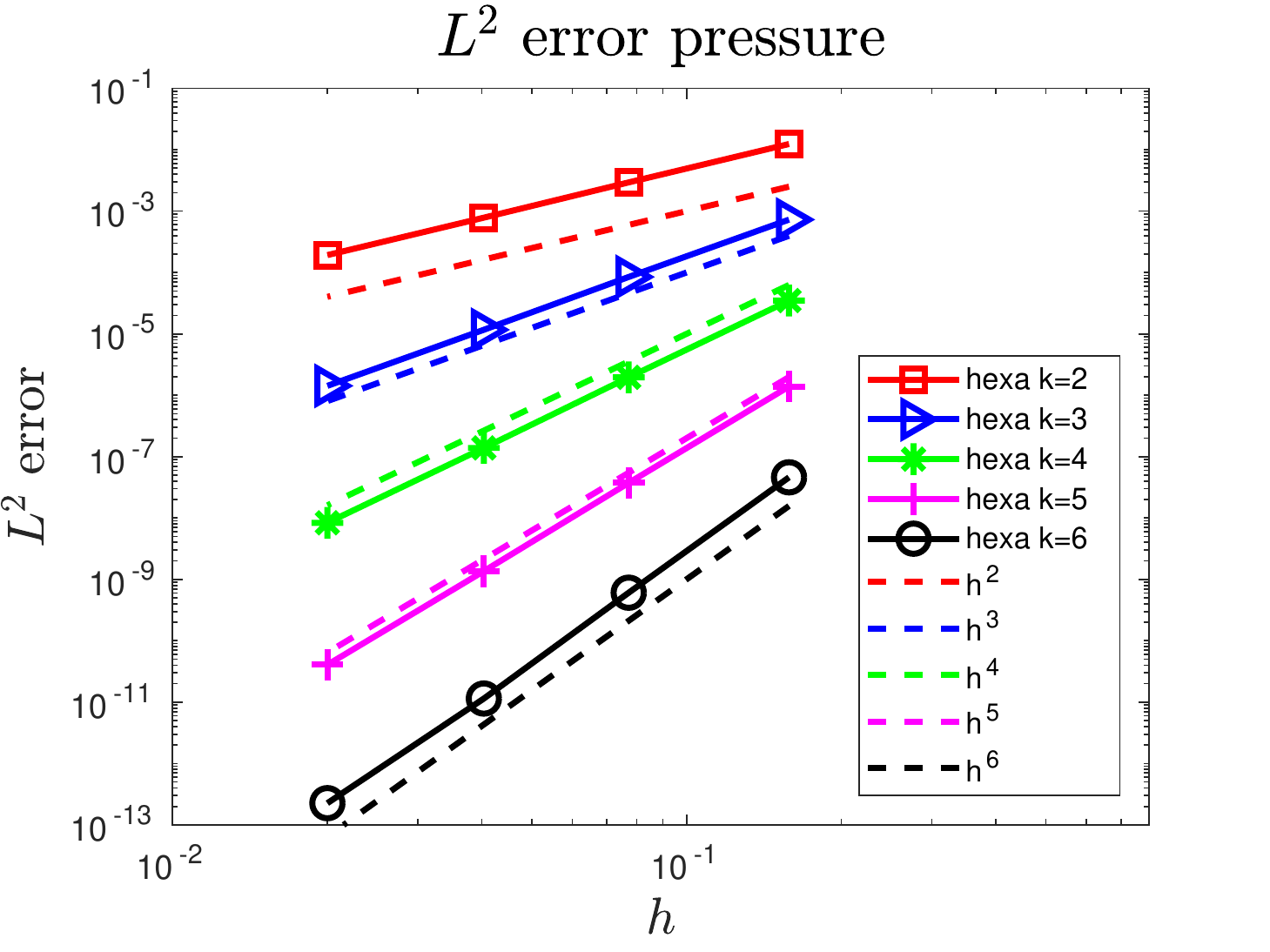} \\
\end{tabular}
\end{center}
\caption{Darcy Problem. Convergence lines for the \texttt{hexa} meshes with $k=2,3,4,5$ and 6.}
\label{fig:hexaDarcy}
\end{figure}

\begin{figure}[!htb]
\begin{center}
\begin{tabular}{cc}
\includegraphics[width=\graphsize\textwidth]{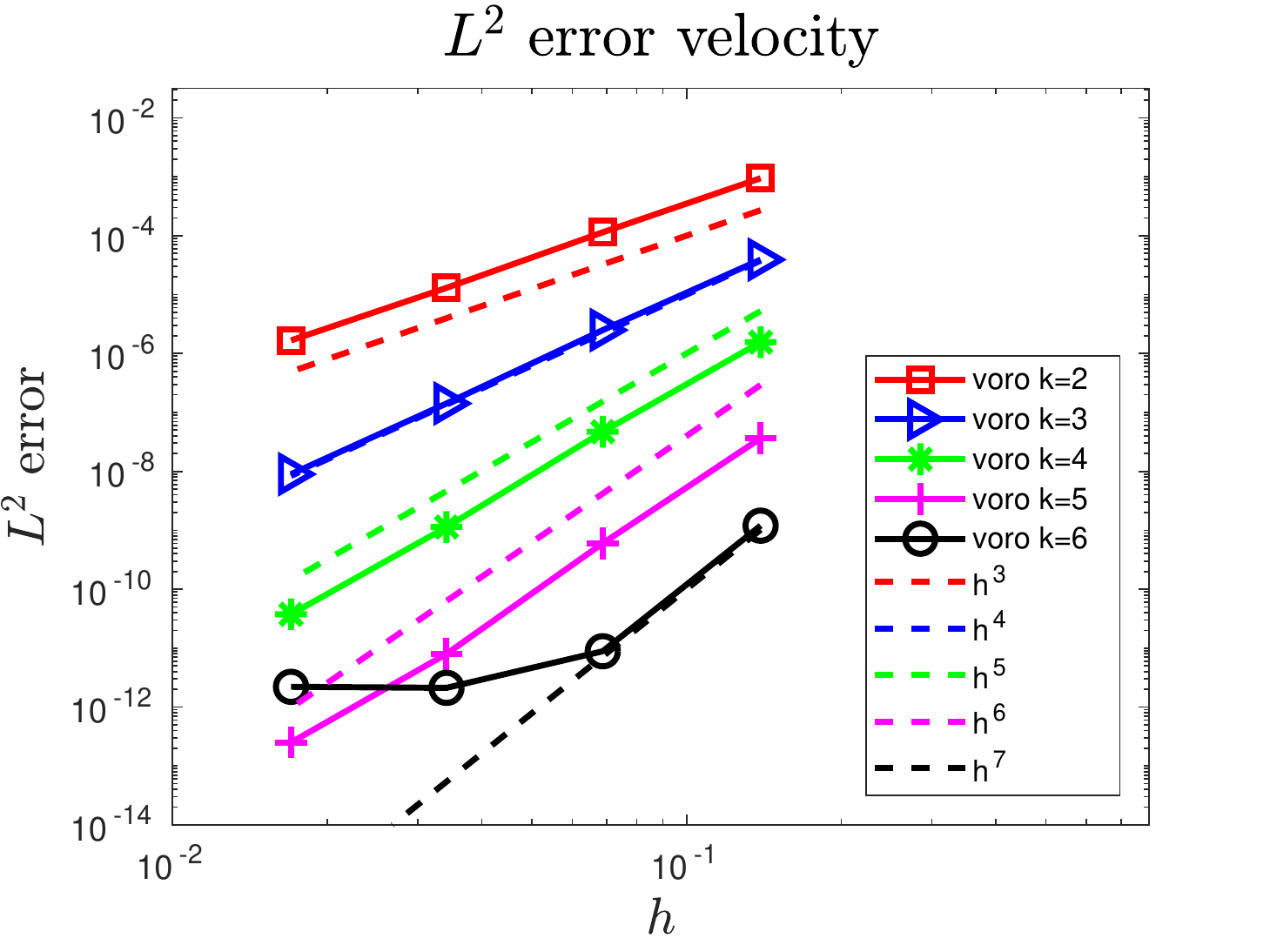} &
\includegraphics[width=\graphsize\textwidth]{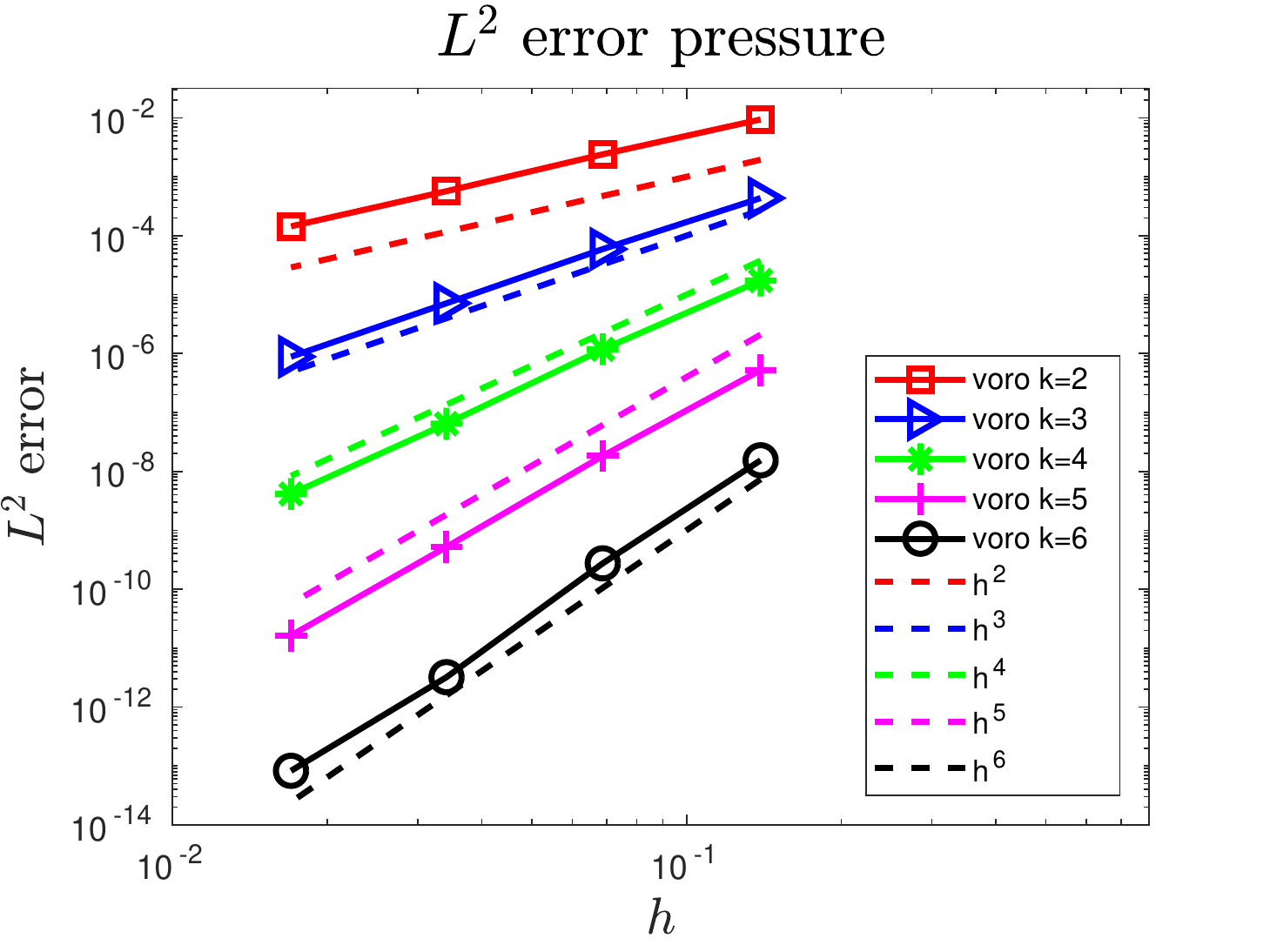} \\
\end{tabular}
\end{center}
\caption{Darcy Problem. Convergence lines for the \texttt{voro} meshes with $k=2,3,4,5$ and 6.}
\label{fig:voroDarcy}
\end{figure}

\subsection{Navier-Stokes Equations}\label{sub:navierStokes}

In this subsection we are solving the Navier-Stokes problem (in the ``gradient form'')
\begin{equation}
\left\{
\begin{aligned}
& \text{find $(\uu_h,\, p_h) \in \VV_h \times Q_h$ s.t.}
\\
& a_h^{\Gr}(\uu_h, \, \vv_h) + c_h(\uu_h; \, \uuh, \, \vv_h) +  b(\vv_h, \, p_h) = (\ff_h, \, \vv_h) 
&\qquad &\text{for all $\vv_h \in \VV_h$,}
\\
& b(\uu_h, \, q_h) = 0
&\qquad &\text{for all $q_h \in Q_h$.}
\end{aligned}
\right.
\label{eqn:navierStokesProb}
\end{equation}
where the bilinear operators $a^{\Gr}_h(\cdot,\cdot)$ and $b(\cdot,\cdot)$ are 
defined in Equations~\eqref{eq:ahGr} and~\eqref{eq:bform}, respectively,
$c_h(\cdot;\,\cdot,\,\cdot)$ is the trilinear operator defined in Equation~\eqref{eq:c_h} and the discrete right-hand side is defined in \eqref{eq:fh}.
Also in this case we fix the load and the boundary condition in such a way that the 
exact solution related to \eqref{eqn:navierStokesProb} is 
$$
\uu(x,\,y) = \left(
\begin{array}{r}
-0.5\,\cos(x)\,\cos(x)\,\cos(y)\,\sin(y)\\
 0.5\,\cos(y)\,\cos(y)\,\cos(x)\,\sin(x)
\end{array}
\right) \,,
\qquad
\p(x,\,y) = \sin(x)-\sin(y)\,.
$$

In Figures~\ref{fig:quadNavierStokes},~\ref{fig:hexaNavierStokes} and~\ref{fig:voroNavierStokes}, 
we show the convergence lines for $k=2,3,4$ and 5.
For \texttt{quad} meshes the trend of such errors is the expected one, see Figure~\ref{fig:quadNavierStokes}.
In the case of the sets of \texttt{hexa} and \texttt{voro} meshes,
we recover the expected trend for $k=2,3$ and 4.
When we consider a degree $k=5$, the last part of the convergence lines do not follow the theoretical trend.
This fact is not so evident for the $\text{\texttt{error}}(\uu, H^1)$, but it becomes clearer in $\text{\texttt{error}}(\p, L^2)$.
Such bad behaviour it is not addicted to the robustness of the virtual element method, 
but it is due to the machine precision.
The error $\text{\texttt{error}}(\uu, L^2)$ has the expected trend,
but it suffers for $k=5$ at the last step in a similar way as $\text{\texttt{error}}(\uu, H^1)$.

\begin{figure}[!htb]
\begin{center}
\begin{tabular}{ccc}
\includegraphics[width=\graphsize\textwidth]{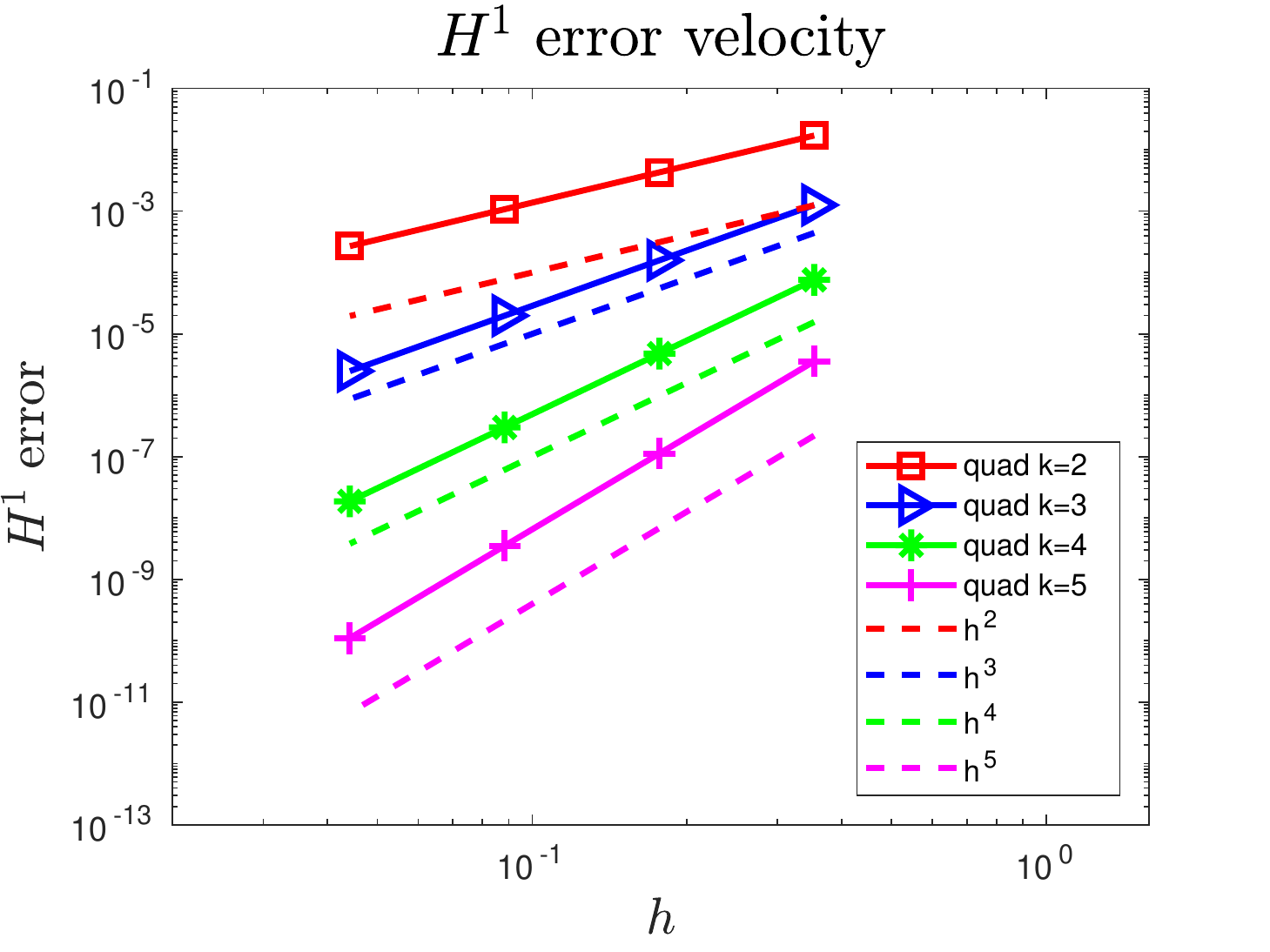} &
\includegraphics[width=\graphsize\textwidth]{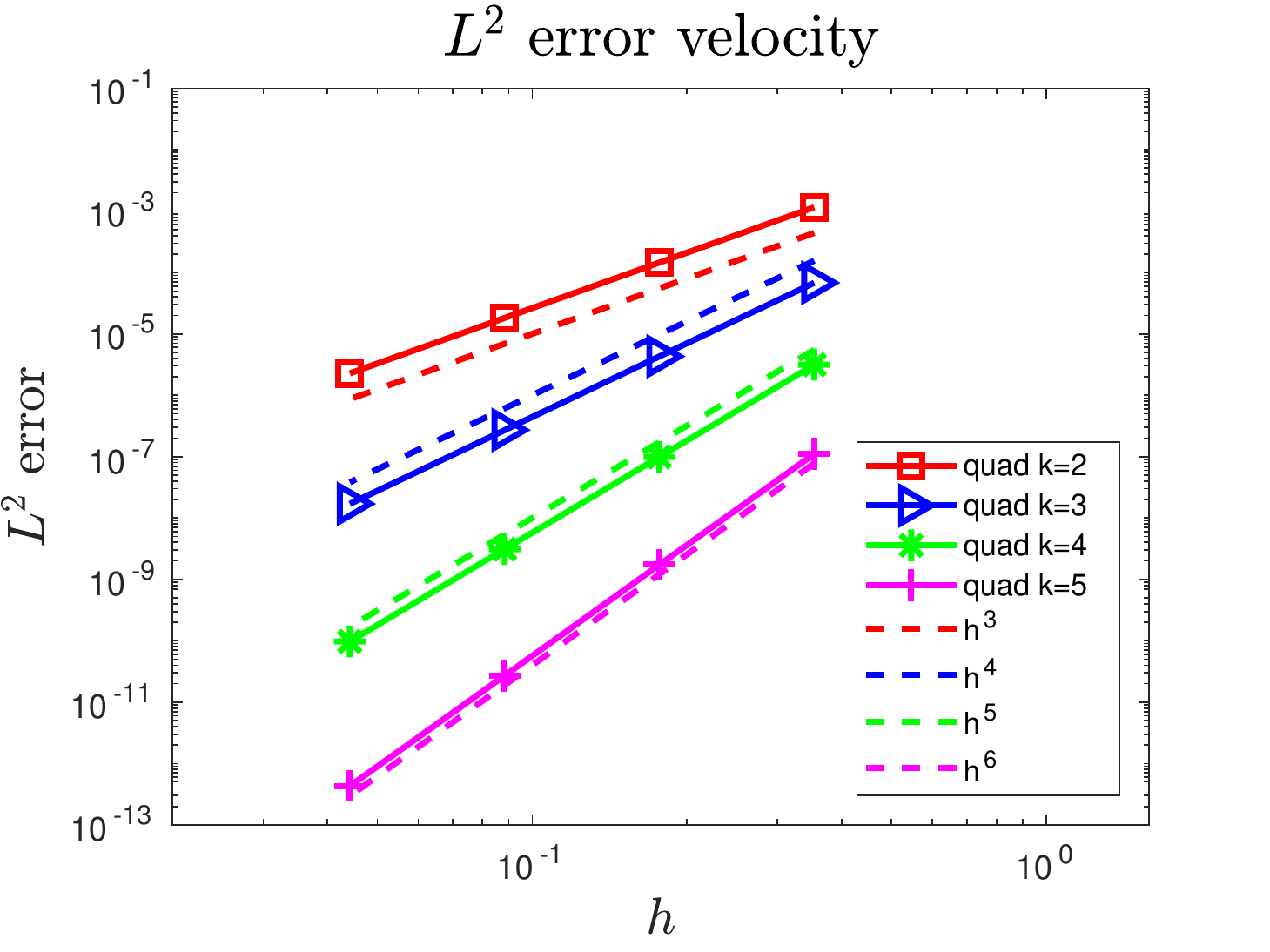} \\
\end{tabular}
\end{center}
\caption{Navier--Stokes Problem. Convergence lines for the \texttt{quad} meshes with $k=2,3,4$ and 5.}
\label{fig:quadNavierStokes}
\end{figure}

\begin{figure}[!htb]
\begin{center}
\begin{tabular}{cc}
\includegraphics[width=\graphsize\textwidth]{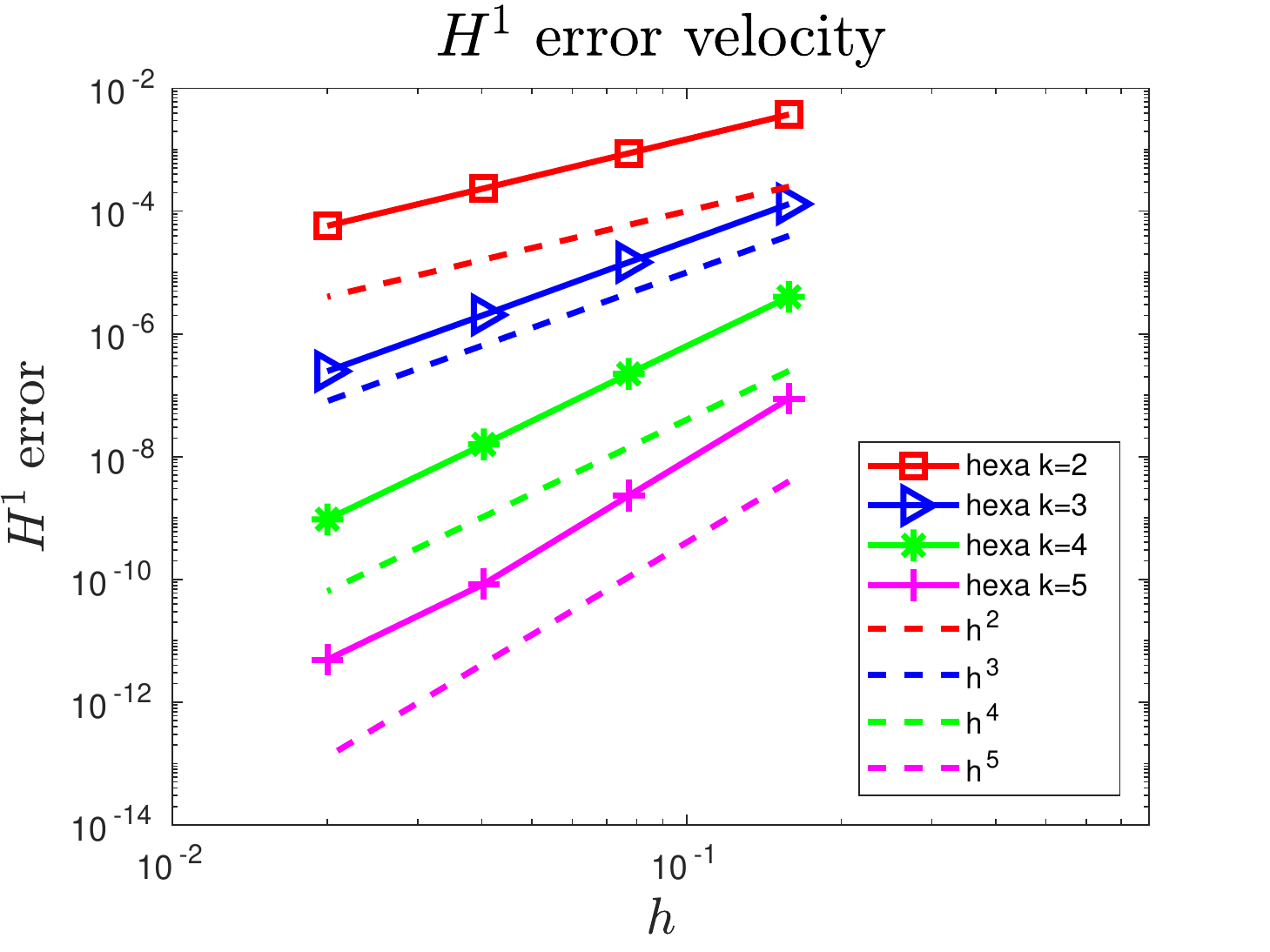} &
\includegraphics[width=\graphsize\textwidth]{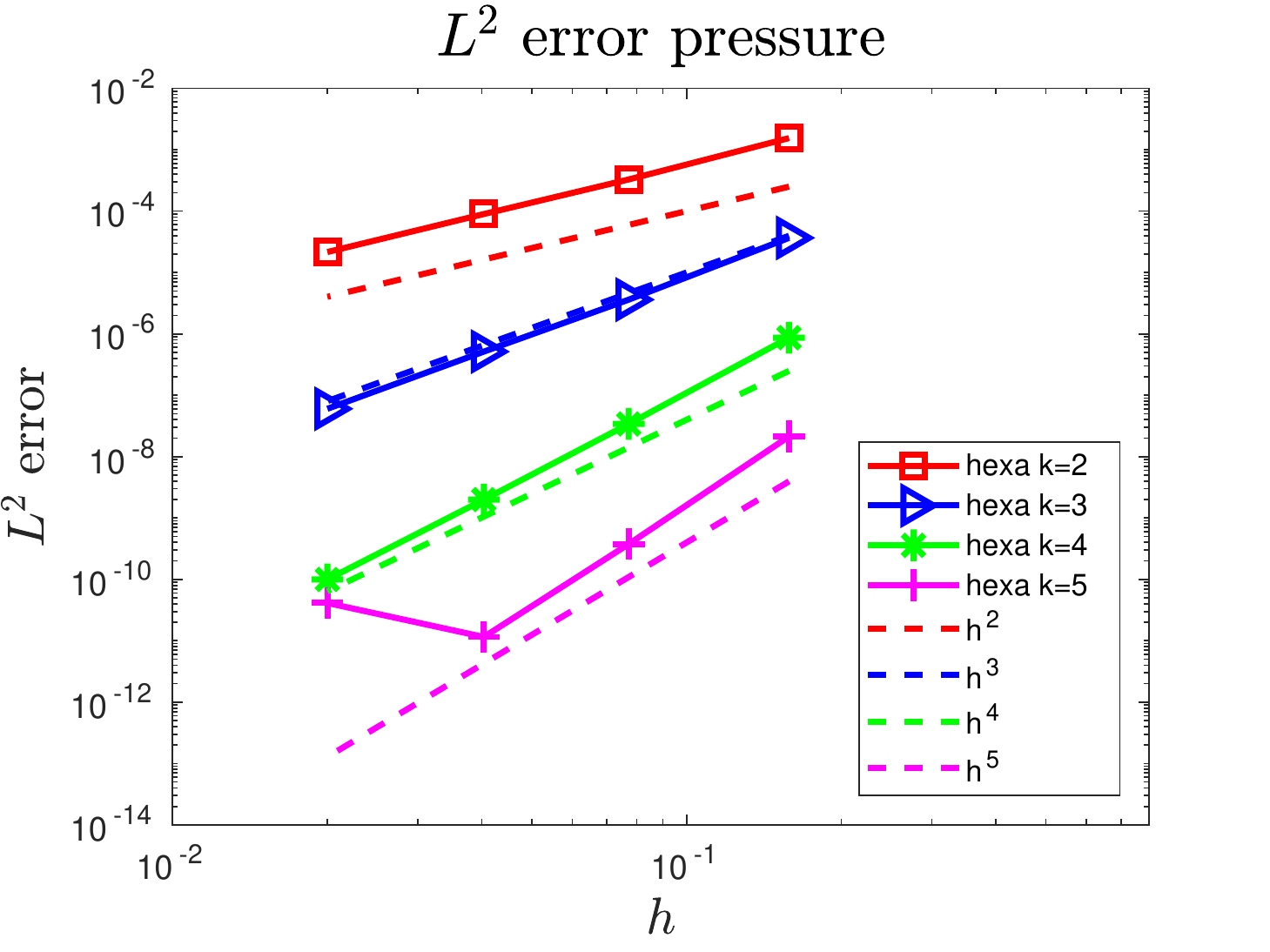} \\
\end{tabular}
\end{center}
\caption{Navier--Stokes Problem. Convergence lines for the \texttt{hexa} meshes with $k=2,3,4$ and 5.}
\label{fig:hexaNavierStokes}
\end{figure}

\begin{figure}[!htb]
\begin{center}
\begin{tabular}{cc}
\includegraphics[width=\graphsize\textwidth]{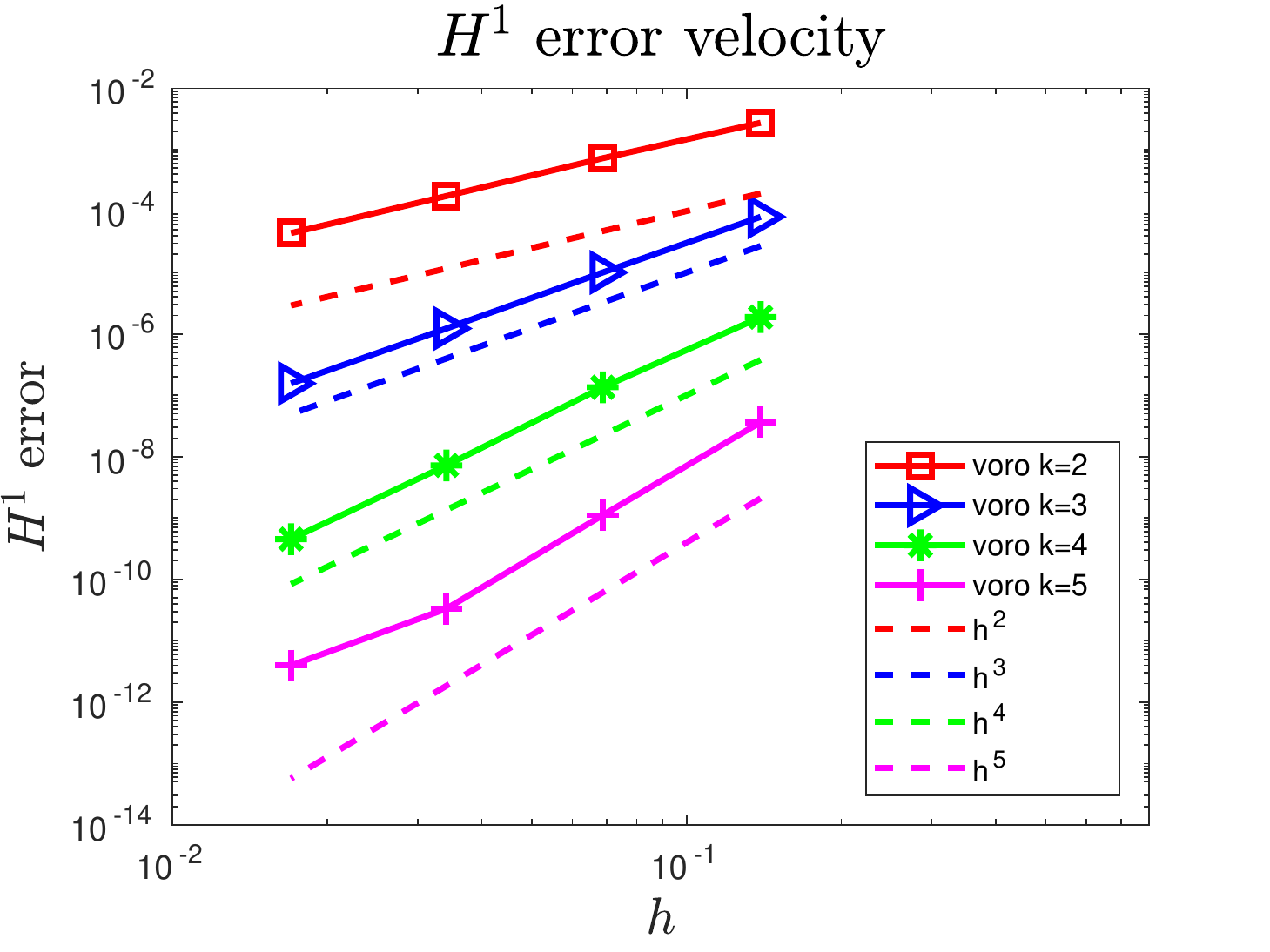} &
\includegraphics[width=\graphsize\textwidth]{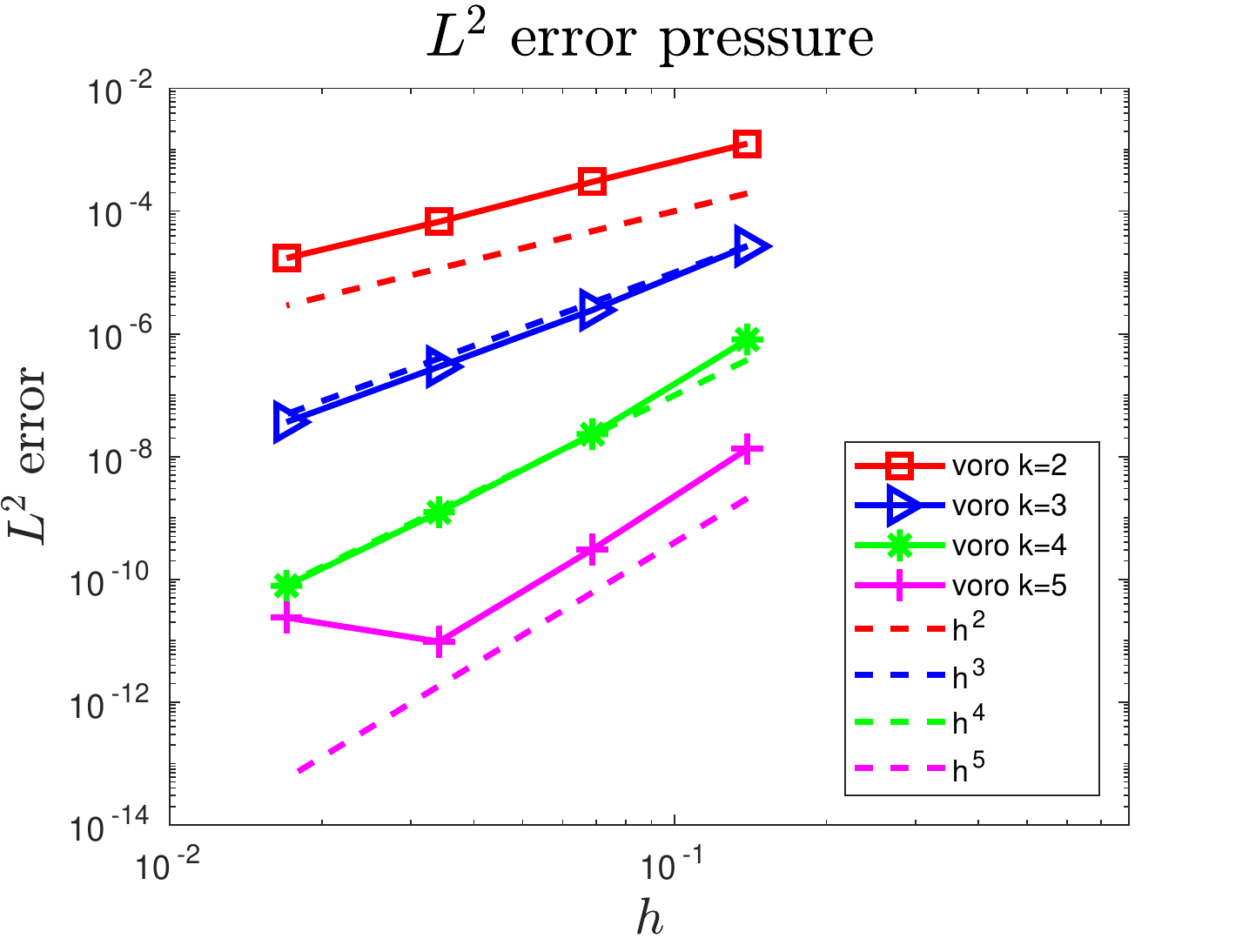} \\
\end{tabular}
\end{center}
\caption{Navier--Stokes Problem.  Convergence lines for the \texttt{voro} meshes with $k=2,3,4$ and 5.}
\label{fig:voroNavierStokes}
\end{figure}

\section{Conclusion}

In this paper we focus on the technical aspects of VEM
when we are considering partial differential equations in mixed form.
More specifically, we gave the essential ``bricks'' 
to make both projectors and differential operators starting from the proposed virtual element spaces.
This deep analysis allowed us to manage high VEM approximation order and
solve a wide variety of problems (Stokes, Brinkman, Darcy and Navier--Stokes).
Numerical results show that VEM are particularly robust with high-order,
since we reach error values close to the machine precision when we are taking high degree and fine meshes.

\section*{Acknowledgments}

The authors would like to acknowledge INDAM-GNCS for the support.
Moreover they would like to thank Lourenco Beir{\~a}o~da Veiga and Alessandro Russo for 
many helpful discussions and suggestions.

\bibliographystyle{plain}
\bibliography{VEM}

\end{document}